\documentclass{article}  
\usepackage{floatrow}
\newfloatcommand{capbtabbox}{table}[][\FBwidth]
\usepackage[pdftex,
bookmarksnumbered,
bookmarksopen,
pagebackref,
colorlinks,
citecolor=blue,
linkcolor=blue,]{hyperref}
\usepackage{amsmath}
\usepackage{amssymb}
\usepackage{mathrsfs}
\usepackage{amsfonts}
\usepackage{amsthm}
\usepackage{algorithm}
\usepackage{algorithmic}
\usepackage{epsfig}
\usepackage{graphicx}
\usepackage{subfigure}
\usepackage{epstopdf}
\usepackage{color}
\usepackage{booktabs}
\usepackage{listings}
\usepackage{boxedminipage}
\usepackage{stmaryrd}
\usepackage{cite}
\usepackage{pgfplots}
\usepackage{tikz}
\usepackage{tkz-tab}
\usepackage{relsize}
 \usepackage{lscape}
\usepackage[top=1.3in, bottom=1.5in, left=1.5in, right=1.5in]{geometry}
\usepackage{xcolor}

\newtheorem{theorem}{Theorem}[section]

\newtheorem{proposition}{Proposition}[section]

\newtheorem{lemma}{Lemma}[section]

\newtheorem{remark}{Remark}[section]

\newcommand{\T}{\ensuremath{ \mathbb R^{n_1\times \cdots\times n_d}   }}

\newcommand{\bigxiaokuohao}[1]{\ensuremath{ \left(  #1 \right) }}      
\newcommand{\bigjueduizhi}[1]{\ensuremath{ \left|  #1 \right| }}

\newcommand{\bigfnorm}[1]{\ensuremath{ \left\|   #1 \right\|_F }}    
\newcommand{\bignorm}[1]{\ensuremath{ \left\|   #1 \right\|  }}        
\newcommand{\normalnorm}[1]{\ensuremath{  \|   #1  \|  }}     
\newcommand{\bigzeronorm}[1]{\ensuremath{ \left\|   #1 \right\|_0  }}

\newcommand{\innerprod}[2]{\ensuremath{ \left\langle   #1 , #2\right\rangle }}

\newcommand{\truncatedvector}[2]{\ensuremath{     \left[#1 \right]^{\downarrow,#2}  }}

	\definecolor{darkgray}{rgb}{0.66, 0.66, 0.66}

\newenvironment{mytabular}{\bgroup\tiny\tabular}{\endtabular\egroup}
\newenvironment{mytabular1}{\bgroup\footnotesize\tabular}{\endtabular\egroup}

\title{Several Approximate  Algorithms for Sparse Best Rank-$1$ Approximation to Higher-Order Tensors}

\author{Xianpeng Mao$^*$ \and Yuning Yang\thanks{College of Mathematics and Information Science, Guangxi University, Nanning, 530004, China.} \thanks{Corresponding author. Email:  yyang@gxu.edu.cn.}}

\begin{document} %\large
\maketitle

%\tableofcontents
%\newpage

\begin{abstract}

	Sparse tensor best rank-1 approximation (BR1Approx), which is a sparsity generalization of the dense tensor BR1Approx, and is a higher-order extension of the sparse matrix BR1Approx, is one of the most important problems in sparse tensor decomposition and related problems arising from statistics and machine learning. By exploiting the multilinearity as well as the sparsity structure of 	the  problem, four polynomial-time approximation algorithms are proposed, which are easily implemented, of low computational complexity, and can serve as initial procedures for iterative algorithms. In addition, theoretically guaranteed %worst-case 
	approximation lower  bounds are derived for all the algorithms.  We provide numerical experiments on synthetic and real data to illustrate the efficiency and effectiveness of the  proposed algorithms; in particular, serving as initialization procedures, the approximation algorithms can help in improving the solution quality of iterative algorithms while reducing the computational time. 
	
	%\noindent {\bf  AMS subject classifications.} 90C26, 15A18, 15A69, 41A50
	\hspace{2mm}\vspace{3mm}
	
\end{abstract}

\thispagestyle{plain} \markboth{X. Mao \and Y. Yang}{Approximation Algorithms for Sparse Tensor BR1Approx}

\section{Introduction}

In the big data era, people often face intrinsically multi-dimensional, multi-modal and multi-view
data-sets that are too complex to be processed and analyzed by traditional   data mining tools based on vectors or matrices. Higher-order tensors (hypermatrices)  naturally can represent such complex data sets. Tensor decomposition tools, developed for understanding tensor-based data sets, have shown the power in various fields such as signal processing, image processing,   statistics, and machine learning; see the surveys \cite{kolda2010tensor,comon2014tensors,comon2009tensor,cichocki2015tensor,sidiropoulos2017tensor,cichocki2014era}.

In high dimensional data-sets, another structure that cannot be ignored is the sparsity.    Sparsity tensor examples come from     clustering problems, online advertising, web link analysis, ranking of authors based on citations   \cite{sun2017provable,sun2019dynamic,papalexakis2012k,kolda2005higher,ng2011multirank}, and so on.    
%Such relations can be modeled as a third-order tensor, whose $(i_1,i_2,i_3)$-th entry is nonzero if the $i_1$-th author's paper is cited by the $i_2$-th author's paper, and the two papers lie in  the same $i_3$-th category, and is zero otherwise. Such a tensor is also sparse.
Therefore, recent advances incorporate   sparsity into tensor decomposition models and tensor-based statistics and machine learning  problems \cite{chi2012tensors,papalexakis2012k,sun2017provable,allen2012sparse,madrid2017tensor,sun2017store,zhang2019optimal,wang2020sparse}; just to name a few.  % 
As pointed out by Allen \cite{allen2012sparse}, introducing sparsity into tensor problems is desirable for feature selection, for   compression, for statistical consistency, and for better visualization and interpretation of the data analysis results. 

In (dense) tensor decomposition and related tensor models and problems, the best rank-1 approximation   (BR1Approx) is one of the most important and fundamental problems \cite{de2000on,qi2018tensor}. Just as the dense counterpart, the sparse tensor BR1Approx     serves as a keystone in the computation of sparse tensor decomposition and related sparse tensor models  \cite{sun2017provable,sun2017store,allen2012sparse,wang2020sparse,papalexakis2012k}. Roughly speaking, the sparse tensor BR1Approx is to find a projection of a given data tensor onto the set of sparse rank-1 tensors in the sense of Euclidean distance.
%, where the sparsity level of each latent factors of the sparse rank-1 tensors is given in advance. 
This is equivalent to maximizing a multilinear function over both unit sphere and $\ell_0$ constraints; mathematical models will be detailed in Section \ref{sec:model}. 
Such a problem    also closely connects to the sparse tensor spectral norm defined in \cite{sun2017provable,sun2017store}. % On the other hand, just as that the sparse canonical correlation analysis (CCA) can be casted as a sparse matrix BR1Approx, it is also  straightforward to cast a higher-order CCA to a sparse tensor BR1Approx \cite{luo2015tensor,min2019tensor}. 

For (dense) tensor BR1Approx, several methods have been proposed; e.g.,  power methods \cite{kolda2011shifted,de2000on}, approximation algorithms  \cite{he2010approximation,zhang2012cubic,he2012probability}, and convex relaxations \cite{nie2014semidefinite,jiang2015tensor,yang2016rank}. For sparse matrix BR1Approx, solution methods have been studied extensively in the context of sparse PCA and sparse SVD \cite{zou2006sparse,witten2009a}; see, e.g.,  iterative   methods \cite{luss2013conditional,witten2009a},  approximation algorithms \cite{chan2016approximability,d2014approximation}, and semidefinite relaxation \cite{aspremont2007a}; just to name a few. For sparse tensor BR1Approx, in the context of sparse tensor decomposition, Allen \cite{allen2012sparse} first studied models and iterative algorithms based on $\ell_1$ regularization in the literature, whereas   Sun et al. \cite{sun2017provable} developed $\ell_0$-based models and algorithms, and analyzed their statistical performance.   Wang et al. \cite{wang2020sparse}    considered another $\ell_1$ regularized model and algorithm, which is different from \cite{allen2012sparse}.  In the study of co-clustering, Papalexakis et al. \cite{papalexakis2012k} proposed   alternating minimization methods for nonnegative sparse tensor BR1Approx. 

For  nonconvex and NP-hard problems, approximation algorithms are nevertheless encouraged. However, in the context of sparse tensor BR1Approx problems, little attention was paid to this type of algorithms. To fill this gap, 
by fully exploiting the multilinearity and sparsity of the model, 
we   develop  four polynomial-time approximation algorithms, some extending their matrix or dense tensor counterparts; in particular, the last algorithm, which is the most efficient one,  is even new when reducing to the matrix or dense tensor cases. %  where the first two can be seen as higher-order generalizations and modifications of those for sparse matrix PCA \cite{chan2016approximability}, and the third one is a sparse extension of that for dense tensor BR1Approx \cite{he2010approximation}. 
The proposed algorithms are easily implemented, and the  computational complexity is not high: the most expensive execution is, if necessary, to only   compute the largest singular vector pairs of certain matrices. Therefore, the algorithms are able to  serve as initialization procedures for iterative algorithms. Moreover, for each algorithm, we derive   theoretically   guaranteed %worst-case
approximation lower bounds. Experiments on synthetic as well as real data show the usefulness of the introduced algorithms.% their own, and the approximation algorithms  can help in improving the performance of iterative algorithms. 
%, which recover their matrix   or non-sparse counterparts. %Our algorithms are inspired by and can be regarded as non-trivial extensions of \cite{chan2016approximability,he2010approximation}; specifically, the first two can be seen as higher-order generalizations and modifications of those for sparse matrix PCA \cite{chan2016approximability}, and the third one is a sparse extension of that for dense tensor BR1Approx \cite{he2010approximation}.  

The rest of this work is organized as follows. Sect. \ref{sec:model} introduces the sparse tensor BR1Approx model. Approximation algorithms are presented in Sect. \ref{sec:approx_alg}. Numerical results are illustrated in Sect. \ref{sec:numer_experiments}. Sect. \ref{sec:conclusions} draws some conclusions.

\section{Sparse Tensor Best Rank-1 Approximation} \label{sec:model}
Throughout this work, vectors are written as  $(\mathbf x,\mathbf y,\ldots)$, matrices
correspond to  $(A,B,\ldots)$, and tensors are
written as  $(\mathcal{A}, \mathcal{B},
\cdots)$. $\mathbb R^{n_1\times \cdots\times n_d}$ denotes the space of $n_1\times\cdots\times n_d$ real tensors. 
For two tensors $\mathcal A,\mathcal B$ of the same size, their inner product $\langle \mathcal A,\mathcal B\rangle$ is given by
the sum of entry-wise product. 
The Frobenius  norm of $\mathcal A$ is defined by $\|\mathcal A\|_F = \langle\mathcal A,\mathcal A\rangle^{1/2}$.
%the spectral norm of $\bignorm{\mathcal A}_2$ is the leading singular value of $\mathcal A$. 
$\circ$ denotes the outer product; in particular, for $\mathbf x_j\in\mathbb R^{n_j}$, $j=1,\ldots,d$, $\mathbf x_1\circ\cdots\circ\mathbf x_d$ denotes a rank-1 tensor in $\T$.  $\bigzeronorm{\mathbf x}$ represents the number of  nonzero entries of a vector $\mathbf x$.

Given  $\mathcal A \in\T$ with $d\geq 3$, the tensor BR1Approx consists of finding a set of vectors $\mathbf x_1,\ldots,\mathbf x_d$, such that
\begin{equation}
	\label{prob:str1approx_dense}
	\min_{\lambda\in\mathbb R,\mathbf x_j\in\mathbb R^{n_j},1\leq j\leq d}~\bigfnorm{\lambda\cdot\mathbf x_1\circ\cdots\circ \mathbf x_d- \mathcal A}^2~{\rm s.t.} \bignorm{\mathbf x_j} = 1,~1\leq j\leq d,
\end{equation}
%namely, it amounts to finding a projection of $\mathcal A$ onto the set of rank-1 tensors. 
When $\mathcal A$ is sparse, it may be necessary to also investigate the sparsity of the latent factors $\mathbf x_j$, $1\leq j\leq d$. Assume that the true sparsity level of each latent factors is known a prior, or can be estimated; then the sparse tensor BR1Approx can be modeled as follows \cite{sun2017provable}:
\begin{equation}
	\label{prob:str1approx_org}
	\min_{\lambda\in\mathbb R, \mathbf x_j\in\mathbb R^{n_j},1\leq j\leq d}~\bigfnorm{\lambda\cdot\mathbf x_1\circ\cdots\circ \mathbf x_d- \mathcal A}^2~{\rm s.t.} \bignorm{\mathbf x_j}=1,~\bigzeronorm{\mathbf x_j}\leq r_j,1\leq j\leq d,
\end{equation}
where   $1\leq r_j\leq n_j$ are positive integers standing for the sparsity level.  Allen \cite{allen2012sparse} and Wang et al. \cite{wang2020sparse} proposed $\ell_1$ regularized models for sparse tensor BR1Approx problems.

Since $\mathbf x_j$'s are normalized in \eqref{prob:str1approx_org}, we have 
\begin{eqnarray*}
	\bigfnorm{\lambda\mathbf x_1\circ\cdots\circ \mathbf x_d - \mathcal A}^2 &=& \bigfnorm{\mathcal A}^2 - 2\lambda\innerprod{\mathcal A}{\mathbf x_1\circ\cdots\circ\mathbf x_d} + \lambda^2\prod^d_{j=1}\nolimits\bignorm{\mathbf x_j}^2\\
	& =& \bigfnorm{\mathcal A}^2 - 2\lambda\innerprod{\mathcal A}{\mathbf x_1\circ\cdots\circ\mathbf x_d} + \lambda^2,
\end{eqnarray*}  
minimizing which with respect to $\lambda$ gives $\lambda = \innerprod{\mathcal A}{\mathbf x_1\circ\cdots\circ\mathbf x_d}$, and so
\[
\bigfnorm{\lambda\mathbf x_1\circ\cdots\mathbf x_d - \mathcal A}^2 =\bigfnorm{\mathcal A}^2 - \innerprod{\mathcal A}{\mathbf x_1\circ\cdots\circ\mathbf x_d}^2.
\]
\color{black}Due to the multilinearity of $\innerprod{\mathcal A}{\mathbf x_1\circ\cdots\circ\mathbf x_d}$, $\innerprod{\mathcal A}{\mathbf x_1\circ\cdots\circ\mathbf x_d}^2$ is maximized if and only if $\innerprod{\mathcal A}{\mathbf x_1\circ\cdots\circ\mathbf x_d}$ is maximized. \color{black}
Thus \eqref{prob:str1approx_org} can be equivalently recast as
\begin{equation}
	\label{prob:str1approx_org_max}
	\boxed{
		\max_{\mathbf x_j,1\leq j\leq d}\nolimits~ \innerprod{\mathcal A}{\mathbf x_1\circ\cdots\circ\mathbf x_d}~{\rm s.t.}~ \bignorm{\mathbf x_j}=1, \bigzeronorm{\mathbf x_j}\leq r_j, 1\leq j\leq d.}
\end{equation}
This is the main model to be focused on. When $r_j=n_j$,  \eqref{prob:str1approx_org_max} boils down exactly to the tensor singular value problem \cite{lim2005singular}. Thus \eqref{prob:str1approx_org} can be regarded as a sparse tensor singular value problem.

When $r_j=n_j$, \eqref{prob:str1approx_org_max} is already NP-hard in general \cite{hillar2013most}; on the other hand, when $d=2$ and $\mathbf x_1=\mathbf x_2$, it is also NP-hard \cite{magdon2017np}.  Therefore, we may deduce that \eqref{prob:str1approx_org_max} is also NP-hard, whose NP-hardness comes from two folds: the multilinearity of the objective function, and the sparsity constraints. In view of this, approximation algorithms for solving \eqref{prob:str1approx_org_max} are   necessary.

In the rest of this work, to simplify notations, we denote
\[
\mathcal A\mathbf x_1\cdots\mathbf x_d:=\innerprod{\mathcal A}{\mathbf x_1\circ\cdots\circ \mathbf x_d}.
\]
In addition, we also use the following notation to denote the partial gradients of $\mathcal A\mathbf x_1 \cdots  \mathbf x_d$ with respect to $\mathbf x_j$:
\begin{equation*}
	\mathcal A\mathbf x_1\cdots\mathbf x_{j-1}\mathbf x_{j+1}\cdots\mathbf x_d ~=~ \nabla_{\mathbf x_j}\mathcal A\mathbf x_1\cdots\mathbf x_d \in\mathbb R^{n_j},~1\leq j\leq d.
	%\mathcal A\mathbf x_2\cdots\mathbf x_d&:=&\nabla_{\mathbf x_1}\mathcal A\mathbf x_1\cdots\mathbf x_d,\\
	%\mathcal A\mathbf x_1\mathbf x_3\cdots\mathbf x_d&:=&\nabla_{\mathbf x_2}\mathcal A\mathbf x\cdots\mathbf x_d,\\
	%&\vdots&,\\
	%\mathcal  A\mathbf x_1\cdots\mathbf x_{d-1}&:=&\nabla_{\mathbf x_d}\mathcal A\mathbf x_1\cdots\mathbf x_d.
\end{equation*}
For example, $(\mathcal A\mathbf x_2\cdots\mathbf x_d)_{i_1} = \sum^{n_2,\ldots,n_d}_{i_2=1,\ldots,i_d=1}\mathcal A_{i_1i_2\cdots i_d}\mathbf x_{2,i_2}\cdots\mathbf x_{d,i_d}$, where we write $\mathbf x_j := [x_{j,1},\ldots,x_{j,n_j}]^\top$.  The partial Hessian of $\mathcal A\mathbf x_1\cdots\mathbf x_d$ with respect to $\mathbf x_{d-1}$ and $\mathbf x_d$ is denoted as:
\[
\mathcal A\mathbf x_1\cdots\mathbf x_{d-2} = \nabla_{\mathbf x_{d-1},\mathbf x_d}\mathcal A\mathbf x_1\cdots \mathbf x_d \in\mathbb R^{n_{d-1}\times n_d},
\]
with $(\mathcal A\mathbf x_1\cdots\mathbf x_{d-2})_{i_{d-1},i_d} = \sum^{n_1,\ldots,n_{d-2}}_{i_1=1,\ldots,i_{d-2}=1}\mathcal A_{i_1\cdots i_{d-1}i_d} \mathbf x_{1,i_1}\cdots \mathbf x_{d-2,i_{d-2}}$.

%Denote $I_j$ as an index set with $I_j\subseteq \{1,\ldots, n_j\}$, $1\leq j\leq d$; $I_j$ does not contain repeated elements. 
%Define $\mathcal A_{I_1\cdots i_d}\in\mathbb R^{r_1\times \cdots \times r_d}$ as a subtensor of $\mathcal A$, given by the Matlab operation $\mathcal A_{I_1\cdots I_d} = \mathcal A(I_1,\ldots,I_d)$. 

%\begin{equation}
%\label{prob:str1approx_org_subtensor}
%\max_{\mathbf x_j,I_j}~ {\mathcal A_{I_1\cdots I_d}}{\mathbf y_1\cdots\mathbf y_d}~{\rm s.t.}~\bignorm{\mathbf y_j}=1,\mathbf y_j\in\mathbb R^{r_j},I_j\subseteq\{1,\ldots, n_j\},~1\leq  d.
%\end{equation}
%Here $I_j$ are regarded as variables.
%
%\begin{proposition}
%	\eqref{prob:str1approx_org_max} is equivalent to \eqref{prob:str1approx_org_subtensor}.
%\end{proposition}
%
%\begin{proposition}
%	\eqref{prob:str1approx_org} is NP-hard.
%\end{proposition}
%\begin{proof}
%Let $I_i$, $i=1,2,3$ in \eqref{prob:str1approx_org_subtensor} be fixed. Then it is a tensor singular value problem, which is NP-had \cite{hillar2013most}. Thus \eqref{prob:str1approx_org_subtensor}, and hence \eqref{prob:str1approx_org}, is NP-hard.
%\end{proof}

\section{Approximation Algorithms and  Approximation Bounds}\label{sec:approx_alg}
Four approximation algorithms are proposed in this section, all of which admit theoretical %worst-case
lower bounds.   We first present some preparations used in this section.  For any nonzero   $\mathbf a\in\mathbb R^n$, $1\leq r\leq n$, denote $ \truncatedvector{\mathbf a}{r}\in\mathbb R^n$ as a truncation of $\mathbf a$  as
\[
\truncatedvector{\mathbf a}{r}_i=
\left\{ \begin{array}{lr}
	a_i, & {\rm if}~|  a_i|{\rm ~ is ~one~of~the~} r~{\rm largest~(in~magnitude)~entries~of~} \mathbf a, \\
	0, & {\rm otherwise}.
\end{array} \right. 
\]
In particular, if $  a_{i_1}$, $  a_{i_2}$, $  a_{i_3},\ldots$ are respectively the $r$-, $(r+1)$-, $(r+2)$-, $\ldots$ largest   entries (in magnitude) with $i_1<i_2<i_3<\cdots$, and $|  a_{i_1}|=|  a_{i_2}| = |a_{i_3}| = \cdots$, then we set $\truncatedvector{\mathbf a}{r}_{i_1}=  a_{i_1}$ and $\truncatedvector{\mathbf a}{r}_{i_2} = \truncatedvector{\mathbf a}{r}_{i_3}=\cdots=0$. Thus $\truncatedvector{\mathbf a}{r}$ is uniquely defined.  We can see that $\truncatedvector{\mathbf a}{r}$ is a best $r$-approximation to $\mathbf a$ \cite[Proposition 4.3]{luss2013conditional}, i.e.,
\begin{equation}
	\label{eq:truncation_equivalent}
	\truncatedvector{\mathbf a}{r} \in \arg\min_{\bigzeronorm{\mathbf x}\leq r}\bignorm{\mathbf x-\mathbf a} ~~\Leftrightarrow~~\frac{\truncatedvector{\mathbf a}{r}}{\normalnorm{\truncatedvector{\mathbf a}{r} }} \in\arg\max_{\bignorm{\mathbf x}=1,\bigzeronorm{\mathbf x}\leq r   } \mathbf a^\top\mathbf x.
\end{equation}
It is not hard to see that the following proposition holds.
\begin{proposition}
	\label{prop:lower_bound_truncation_times_org}
	Let $\mathbf a\in\mathbb R^n$, $\mathbf a\neq 0$ and let $\mathbf a^0 = \truncatedvector{\mathbf a}{r}/\normalnorm{\truncatedvector{\mathbf a}{r}}$ with $1\leq r\leq n$. Then 
	\[
	\innerprod{\mathbf a}{\mathbf a^0} \geq \sqrt{\frac{r}{n}}\bignorm{\mathbf a}.  
	\]
\end{proposition}

Let $\lambda_{\max}(\cdot)$ denote the largest singular value of a given matrix. The following lemma is important.
\begin{lemma}
	\label{lem:lower_bound_lem:1}
	Given a nonzero $A\in\mathbb R^{m\times n}$, with $(\mathbf y,\mathbf z)$ being the normalized singular vector pair corresponding to $\lambda_{\max}(A)$. Let $\mathbf z^0 = \truncatedvector{\mathbf z}{r}/\normalnorm{\truncatedvector{\mathbf z}{r}}$ with $1\leq r\leq n$. Then there holds
	\[\small
	\bignorm{A\mathbf z^0} \geq \sqrt{\frac{r}{n}}\lambda_{\max}(A).
	\]
\end{lemma}
\begin{proof}
	From the definition of $\mathbf z$, we see that $\bignorm{A \mathbf z} = \lambda_{\max}(A)$ and $A^\top A\mathbf z = \lambda_{\max}^2(A)\mathbf z$. Therefore,
	\begin{eqnarray*}
		\lambda_{\max}(A)	\bignorm{A\mathbf z^0} &=& \bignorm{A\mathbf z^0}\cdot \bignorm{A {\mathbf z}}\nonumber\\
		&\geq &  \innerprod{A\mathbf z^0}{A {\mathbf z}}\nonumber\\
		&=&  \lambda_{\max}(A) \cdot \innerprod{A\mathbf z^0}{ {\mathbf y}}= \lambda^2_{\max}(A)\cdot \innerprod{\mathbf z^0}{ {\mathbf z}}\nonumber\\
		&\geq & \sqrt{\frac{r }{n }}\lambda^2_{\max}(A),  
	\end{eqnarray*}
	where the   last inequality follows from Proposition \ref{prop:lower_bound_truncation_times_org},  the definition of $\mathbf z^0  $,  and that $\bignorm{ {\mathbf z}}=1$. This completes the proof. 
\end{proof}
\subsection{The first algorithm}\label{sec:approx_alg1}

To illustrate the first algorithm, we denote $\mathbf e_j^{i_j}\in\mathbb R^{n_j}$, $1\leq i_j\leq n_j$, $j=1,\ldots,d$,  as   standard basis vectors in $\mathbb R^{n_j}$. For example, $\mathbf e_2^1$ is a vector  in $\mathbb R^{n_2}$ with the first entry being one and the remaining ones being zero. Denote $\boldsymbol{r}:=(r_1,\ldots,r_d)$; without loss of generality, we assume that $r_1\leq\cdots\leq r_d$. 

\begin{boxedminipage}{0.92\textwidth}\small
	\begin{equation}  \label{proc:init1_order_d}
		\noindent {\rm Algorithm}~ (\mathbf x^0_1,\ldots,\mathbf x^0_d) =  {\rm approx\_alg}(\mathcal A,\boldsymbol{r})
		\tag{A} 
	\end{equation}
	
	1. For each $i_j=1,\ldots,n_j$, $j=1,\ldots,d-1$, compute $\truncatedvector{\mathcal A\mathbf e_1^{i_1}\cdots\mathbf e_{d-1}^{i_{d-1}}}{r_{d}}\in\mathbb R^{n_d}$.
	
	2. Let $(\bar i_1,\ldots,\bar i_{d-1})$ be a tuple of indices such that $$\bignorm{ \truncatedvector{\mathcal A\mathbf e_1^{\bar i_1}\cdots\mathbf e_{d-1}^{\bar i_{d-1}}}{r_d}} = \max_{1\leq i_j\leq n_j,1\leq j\leq d-1} \bignorm{\truncatedvector{\mathcal A\mathbf e_1^{i_1}\cdots\mathbf e_{d-1}^{i_{d-1}}}{r_d}};$$ denote $\bar{\mathbf x}^0_d:= \truncatedvector{\mathcal A\mathbf e_1^{\bar i_1}\cdots\mathbf e_{d-1}^{\bar i_{d-1}}}{r_d}$ and $\mathbf x_d^0:= \bar{\mathbf x}^0_d/\normalnorm{\bar{\mathbf x}^0_d}$.
	
	3. Sequentially update 
	\begin{eqnarray*}
		\bar{\mathbf x}^0_{d-1}&=&\truncatedvector{\mathcal A\mathbf e_1^{\bar i_1}\cdots\mathbf e_{d-2}^{\bar i_{d-2}}\mathbf x^0_d  }{r_{d-1}},\mathbf x^0_{d-1} = \bar{\mathbf x}^0_{d-1}/\bignorm{\bar{\mathbf x}^0_{d-1}},\\
		\bar{\mathbf x}^0_{d-2} &=& \truncatedvector{\mathcal A\mathbf e_1^{\bar i_1}\cdots\mathbf e_{d-3}^{\bar i_{d-3}}  \mathbf x^0_{d-1}\mathbf x^0_d   }{r_{d-2}}, \mathbf x^0_{d-2} = \bar{\mathbf x}^0_{d-2}/\bignorm{\bar{\mathbf x}^0_{d-2}},\\
		& \vdots&\\
		\bar{\mathbf x}^0_{ 1} &=& \truncatedvector{\mathcal A\mathbf x^0_2\cdots  \mathbf x^0_d}{r_{ 1}},\mathbf x^0_{ 1} = \bar{\mathbf x}^0_{ 1}/\bignorm{\bar{\mathbf x}^0_{ 1}}.
	\end{eqnarray*}
	
	4. Return $(\mathbf x^0_1,\ldots,\mathbf x^0_d)$.
\end{boxedminipage}

It is clear that $\mathcal A\mathbf e_1^{i_1}\cdots\mathbf e_{d-1}^{i_{d-1}}$'s are      mode-$d$ fibers of $\mathcal A$. For the definition of fibers, one can refer to \cite{kolda2010tensor}; following the Matlab notation we have $\mathcal A(i_1,\ldots,i_{d-1},:) = \mathcal A\mathbf e_1^{i_1}\cdots\mathbf e_{d-1}^{i_{d-1}}$.  Algorithm \ref{proc:init1_order_d} is a straightforward extension of \cite[Algorithm 1]{chan2016approximability} for sparse symmetric matrix PCA  to higher-order tensor cases.   Intuitively, the first two steps of Algorithm \ref{proc:init1_order_d}  enumerate all the mode-$d$ fibers of $\mathcal A$, such that the select one admits the largest length with respect to  its  largest $r_d$ entries (in magnitude). $\mathbf x^0_d$ is then  given by the normalization of this fiber. Then, according to \eqref{eq:truncation_equivalent}, the remaining $\mathbf x^0_j$   are  in fact obtained by sequentially updated as
\begin{equation}\label{eq:als}
	\mathbf x^0_j \in\arg\max_{\bignorm{\mathbf y}=1,\bigzeronorm{\mathbf y}\leq r_j  } \mathcal A\mathbf e_1^{\bar i_1}\cdots \mathbf e_{j-1}^{\bar i_{j-1}}\mathbf y \mathbf x^{0}_{j+1}\cdots\mathbf x^0_d,~j=d-1,\ldots,1.
\end{equation}

We consider the computational complexity of Algorithm \ref{proc:init1_order_d}. Computing $\mathcal A\mathbf e_1^{i_1}\cdots\mathbf e_{d-1}^{i_{d-1}}$ takes $O(n_d)$ flops, while it takes $O(n_d\log(n_d))$ flops to compute $\truncatedvector{\mathcal A\mathbf e_1^{i_1}\cdots\mathbf e_{d-1}^{i_{d-1}}}{r}$       by using quick-sort. Thus   the first two steps take $O(n_1\cdots n_d\log(n_d) + n_d\log(n_d))$ flops.   
Computing $\mathbf x^0_j,j=1,\ldots,d-1$ respectively takes $O(\prod^j_{k=1}n_k \cdot n_d+n_j\log(n_j)  )$ flops. Thus the total complexity can be roughly estimated as $O(\prod^d_{j=1}n_j\cdot\log(n_d) + \sum^{d}_{j=1}n_j\log(n_j))$. %which is not high, compared with the size of the data tensor.

We first show that Algorithm \ref{proc:init1_order_d} is well-defined.
\begin{proposition}
	\label{prop:welldefined1} If $\mathcal A\neq 0$ and $(\mathbf x^0_1,\ldots,\mathbf x^0_d)$ is generated by Algorithm \ref{proc:init1_order_d}, then $\mathbf x^0_j\neq 0$, $1\leq j\leq d$.
\end{proposition}
\begin{proof}
	Since $\mathcal A$ is a nonzeros tensor, there exists at least one fiber that is not identically zero. Therefore, the definition of $\bar{\mathbf x}^0_d$ shows that $\bar{\mathbf x}^0_d$ is not identically zero, and hence $\mathbf x^0_d$. We also observe that $\innerprod{ \mathcal A\mathbf e_1^{\bar i_1}\cdots\mathbf e_{d-2}^{\bar i_{d-2}}\mathbf x^0_d   }{ \mathbf e^{\bar i_{d-1}}_{d-1}   }=\mathcal A\mathbf e_1^{\bar i_1}\cdots\mathbf e_{d-1}^{\bar i_{d-1}}\mathbf x^0_d = \normalnorm{  \mathcal A\mathbf e_1^{\bar i_1}\cdots\mathbf e_{d-1}^{\bar i_{d-1}}  } >0$, which implies that $\mathcal A\mathbf e_1^{\bar i_1}\cdots\mathbf e_{d-2}^{\bar i_{d-2}}\mathbf x^0_d  \neq 0$,  and so $\mathbf x^0_{d-1}\neq 0$.  \eqref{eq:als}   shows that  $\mathcal A\mathbf e^{\bar i_1}_1\cdots\mathbf e^{\bar i_{d-2}}_{d-2}\mathbf x^0_{d-1}\mathbf x^0_{d} \geq \mathcal A\mathbf e_1^{\bar i_1}\cdots\mathbf e_{d-1}^{\bar i_{d-1}}\mathbf x^0_d  >0$.    Similarly, we can show that $\mathbf x^0_{d-2}\neq 0,\ldots,\mathbf x^0_1\neq 0$.  
\end{proof}

We next first present the approximation bound when $d=3$.   The bound extends   that of \cite{chan2016approximability} to higher-order cases. 
\begin{theorem}
	\label{th:init_theory_bound}
	Denote $v^{\rm opt}$ as the optimal value of \eqref{prob:str1approx_org_max} when $d=3$.
	%	\begin{equation}
		%	\label{prob:sparse_constrained_str1approx}
		%	\max_{\mathbf x\in\mathbb R^{n_1},\mathbf y\in\mathbb R^{n^2},\mathbf z\in\mathbb R^{n^3}}\mathcal A\mathbf x\mathbf y\mathbf z~~{\rm s.t.}~\bignorm{\mathbf x}=\bignorm{\mathbf y}=\bignorm{\mathbf z}=1, \bigzeronorm{\mathbf x}\leq r_1,\bigzeronorm{\mathbf y}\leq r_2,\bigzeronorm{\mathbf z}\leq r_3.
		%	\end{equation}
	Let $(\mathbf x^0_1,\mathbf x^0_2,\mathbf x^0_3)$ be generated by Algorithm \ref{proc:init1_order_d}. Then it holds that
	\[\small
	\mathcal A\mathbf x^0_1\mathbf x^0_2\mathbf x^0_3\geq \frac{v^{\rm opt}}{\sqrt{r_1r_2}}.
	\]
\end{theorem}
\begin{proof}
	Denote $(\mathbf x^*_1,\mathbf x^*_2,\mathbf x^*_3)$ as a maximizer of \eqref{prob:str1approx_org_max}. By noticing that $\bigzeronorm{\mathbf x^*_1}\leq r_1$ and $\mathcal A\mathbf x^*_1\mathbf x^*_2\mathbf x^*_3=\sum_{\{i_1: x^*_{1,i_1}\neq 0\}}^{n_1}  x^*_{1,i_1} \cdot \bigxiaokuohao{ \mathcal A\mathbf e_1^{i_1} \mathbf x^*_2\mathbf x^*_3  } $; recalling that we write $\mathbf x_j := [x_{j,1},\ldots,x_{j,n_j}]^\top$, we have
	\begin{eqnarray}
		v^{\rm opt} &=& \mathcal A\mathbf x^*_1\mathbf x^*_2\mathbf x^*_3= \sum^{n_1}_{\{i_1: x^*_{1,i_1}\neq 0  \}}\nolimits  x^*_{1,i_1}\bigxiaokuohao{\mathcal A\mathbf e_1^{i_1}\mathbf x^*_2\mathbf x^*_3}\nonumber\\
		&\leq& \sqrt{ \sum^{n_1}_{ \{i_1: x^*_{1,i_1}\neq 0  \}}\nolimits \bigxiaokuohao{  x^*_{1, i_1}}^2  }\sqrt{ \sum^{n_1}_{ \{i_1: x^*_{1,i_1}\neq 0  \}}\nolimits\bigxiaokuohao{ \mathcal A\mathbf e_1^{i_1}\mathbf x^*_2\mathbf x^*_3  }^2  }\nonumber\\
		&\leq& \bignorm{\mathbf x^*_1} \cdot \sqrt{r_1} \max_{i_1} \bigjueduizhi{\mathcal A\mathbf e_1^{i_1}\mathbf x^*_2\mathbf x^*_3 }=\sqrt{r_1} \max_{i_1} \bigjueduizhi{\mathcal A\mathbf e_1^{i_1}\mathbf x^*_2\mathbf x^*_3 },\label{eq:th:lower_bound:1}
	\end{eqnarray}
	where the first inequality uses the Cauchy-Schwartz inequality, and the last equality follows from $\bignorm{\mathbf x^*_1}=1$. In the same vein, we have
	\begin{eqnarray*}
		v^{\rm opt}  \leq  \sqrt{r_1}\max_{i_1} \bigjueduizhi{\mathcal A\mathbf e_1^{i_1}\mathbf x^*_2\mathbf x^*_3} 
		\leq  \sqrt{r_1r_2}\max_{i_1,i_2} \bigjueduizhi{\mathcal A\mathbf e_1^{i_1}\mathbf e_2^{i_2}\mathbf x^*_3}. \label{eq:th:lower_bound:2}
	\end{eqnarray*} 
	Assume that $\bigjueduizhi{\mathcal A\mathbf e_1^{\hat i_1}\mathbf e_2^{\hat i_2}\mathbf x^*_3} = \max_{i_1,i_2}\bigjueduizhi{\mathcal A\mathbf e_1^{i_1}\mathbf e_2^{i_2}\mathbf x^*_3}$; denote
	$$\hat{\mathbf x}_3 := \truncatedvector{\mathcal A\mathbf e_1^{\hat i_1}\mathbf e_2^{\hat i_2}}{r_3}\Big /\bignorm{\truncatedvector{\mathcal A\mathbf e_1^{\hat i_1}\mathbf e_2^{\hat i_2}}{r_3}}.$$ \eqref{eq:truncation_equivalent} shows that $\bigjueduizhi{\mathcal A\mathbf e_1^{\hat i_1}\mathbf e_2^{\hat i_2}\mathbf x^*_3 }\leq \bigjueduizhi{\mathcal A\mathbf e_1^{\hat i_1}\mathbf e_2^{\hat i_2}\hat{\mathbf x}_3}$.  	By noticing the  definition of $\mathbf x^0_3$, we then have 
	\begin{eqnarray*}
		\bigjueduizhi{\mathcal A\mathbf e_1^{\hat i_1}\mathbf e_2^{\hat i_2}\mathbf x^*_3 }&\leq& \bigjueduizhi{\mathcal A\mathbf e_1^{\hat i_1}\mathbf e_2^{\hat i_2}\hat{\mathbf x}_3}= \bignorm{\truncatedvector{\mathcal A\mathbf e_1^{\hat i_1}\mathbf e_2^{\hat i_2}}{r_3}}\leq \max_{i_1,i_2} \bignorm{ \truncatedvector{\mathcal A\mathbf e_1^{i_1}\mathbf e_2^{i_2}}{r_3}   } \\
		&=& \mathcal A\mathbf e_1^{\bar i_1}\mathbf e_2^{\bar i_2}\mathbf x^0_3.
	\end{eqnarray*}
	%	where the last equality follows from the definition  of  $\mathbf x^0_3$ in Algorithm \ref{proc:init}. 
	Finally, recalling the definitions of $\mathbf x^0_1$ and $\mathbf x^0_2$ and combining the above pieces, we arrive at
	\[
	v^{\rm opt} \leq \sqrt{r_1r_2}\mathcal A\mathbf e_1^{\bar i_1}\mathbf e_2^{\bar i_2}\mathbf x^0_3\leq \sqrt{r_1r_2}\mathcal A\mathbf e_1^{\bar i_1}\mathbf x^0_{2}\mathbf x^0_3\leq \sqrt{r_1r_2}\mathcal A\mathbf x^0_1\mathbf x^0_2\mathbf x^0_3,
	\]
	as desired.   
\end{proof}

In the same spirit of the proof of   Theorem \ref{th:init_theory_bound}, for general $d\geq 3$, one can show that
\begin{small}
	\begin{eqnarray*}
		v^{\rm opt} &\leq& \sqrt{r_1}\max_{i_1}\nolimits \bigjueduizhi{ \mathcal A\mathbf e_1^{i_1}\mathbf x^*_2\cdots\mathbf x^*_d } \leq \cdots\leq \sqrt{\prod^{d-1}_{j=1}\nolimits r_j}\max_{i_1,\ldots,i_{d-1}}\bigjueduizhi{ \mathcal A\mathbf e_1^{i_1}\cdots\mathbf e_{d-1}^{i_{d-1}}\mathbf x^*_d  }\\
		&\leq& \sqrt{\prod^{d-1}_{j=1}\nolimits r_j}\max_{i_1,\ldots,i_{d-1}} \mathcal A\mathbf x^0_1\cdots\mathbf x^0_d.
	\end{eqnarray*}
\end{small}
%Thus we present the approximation bound for general $d\geq 3$ and omit the proof. 
\begin{theorem}
	\label{th:init_theory_bound_1_general}
	Denote $v^{\rm opt}$ as the optimal value of the problem \eqref{prob:str1approx_org_max}. 
	%	\begin{equation}
		%	\label{prob:sparse_constrained_str1approx_general}
		%	\max_{\mathbf x_j\in\mathbb R^{n_j},1\leq j\leq d}\mathcal A\mathbf x_1\cdots\mathbf x_d~~{\rm s.t.}~\bignorm{\mathbf x_j}=1, \bigzeronorm{\mathbf x_j}\leq r_j,~1\leq j\leq d.
		%	\end{equation}
	Let $(\mathbf x^0_1,\ldots,\mathbf x^0_d)$ be generated by Algorithm \ref{proc:init1_order_d} for a general $d$-th order tensor $\mathcal A$. Then it holds that
	\[\small
	\mathcal A\mathbf x^0_1\cdots\mathbf x^0_d\geq \frac{v^{\rm opt}}{\sqrt{\prod^{d-1}_{j=1}r_j}}.
	\]
\end{theorem}
\subsection{The second algorithm}\label{sec:approx_alg2}
The second algorithm is presented as follows.

\begin{boxedminipage}{0.92\textwidth}\small
	\begin{equation}  \label{proc:init2_order_d}
		\noindent {\rm Algorithm}~ (\mathbf x^0_1,\ldots,\mathbf x^0_d) =  {\rm approx\_alg}(\mathcal A,\boldsymbol{r}) 
		\tag{B}
	\end{equation}
	
	1. For each tuple $(i_1,\ldots,i_{d-2})$, $i_j=1,\ldots,n_{j}$, $1\leq j\leq d-2$, solve the matrix singular value problem $\max_{\normalnorm{\mathbf x_{d-1}}=\normalnorm{\mathbf x_d}=1}  \mathcal A\mathbf e_1^{i_1}\cdots\mathbf e_{d-2}^{i_{d-2}}\mathbf x_{d-1}\mathbf x_d $.
	
	2. Let $(\bar i_1,\ldots,\bar i_{d-2} )$ be the optimal tuple of indices with $(\bar{\mathbf x}_{d-1},\bar{\mathbf x}_d)$ being the optimal solution pair, i.e., $$ \mathcal A\mathbf e_1^{\bar i_1}\cdots\mathbf e_{d-2}^{\bar i_{d-2}}\bar{\mathbf x}_{d-1}\bar{\mathbf x}_d= \max_{1\leq i_j\leq n_j,1\leq j\leq d-2,\normalnorm{\mathbf x_{d-1}}=\normalnorm{\mathbf x_d}=1}  \mathcal A\mathbf e_1^{i_1}\cdots\mathbf e_{d-2}^{i_{d-2}}\mathbf x_{d-1}\mathbf x_{d};$$ 
	denote $ {\mathbf x}^0_d:= \truncatedvector{\bar{\mathbf x}_d}{r_d}/\normalnorm{\truncatedvector{\bar{\mathbf x}_d}{r_d}}$.
	
	3. Sequentially update $\mathbf x^{0}_{d-1},\ldots,\mathbf x^0_1$ as Step 3 of Algorithm \ref{proc:init1_order_d}. 
	
	4. Return $(\mathbf x^0_1,\ldots,\mathbf x^0_d)$.
\end{boxedminipage}

The main difference from Algorithm \ref{proc:init1_order_d} mainly lies in the first step, where Algorithm \ref{proc:init1_order_d} requires to find the   fiber with the largest length with respect to the largest $r_3$ entries (in magnitude), while the first step of Algorithm \ref{proc:init2_order_d} looks for the matrix $\mathcal A\mathbf e_1^{\bar i_1}\cdots\mathbf e_{d-2}^{\bar i_{d-2}} \in\mathbb R^{n_{d-1}\times n_d}$ with the largest spectral radius among all $\mathcal A\mathbf e_1^{i_1}\cdots\mathbf e_{d-2}^{i_{d-2}}$.
%\begin{boxedminipage}{0.97\textwidth}\small
%	\begin{equation}  \label{proc:init2}
	%	\noindent {\rm Algorithm}~ (\mathbf x^0_1,\mathbf x^0_2,\mathbf x^0_3) =  {\rm approx\_alg}(\mathcal A,r_1,r_2,r_3) 
	%	\tag{B0}
	%	\end{equation}
%	
%	1. For each $i_1=1,\ldots,n_1$, solve the matrix singular value problem $\max_{\bignorm{\mathbf y}=\bignorm{\mathbf z}=1}  \mathcal A\mathbf e^1_{i_1}\mathbf x_2\mathbf x_3 $.
%	
%	2. Let $\bar i_1 $ be the optimal index with $(\bar{\mathbf x}_2,\bar{\mathbf x}_3)$ being the optimal solution pair such that $$ \mathcal A\mathbf e^1_{\bar i_1}\bar{\mathbf x}_2\bar{\mathbf x}_3= \max_{i_1,\bignorm{\mathbf y}=\bignorm{\mathbf z}=1}  \mathcal A\mathbf e^1_{i_1}\mathbf y\mathbf z;$$ 
%	denote $ {\mathbf x}^0_3:= \truncatedvector{\bar{\mathbf x}_3}{r_3}/\bignorm{\truncatedvector{\bar{\mathbf x}_3}{r_3}}$.
%	
%	3. Compute $\bar{\mathbf x}^0_2:=\truncatedvector{\mathcal A\mathbf e^1_{\bar i_1}\mathbf x^0_3}{r_2}$, $\mathbf x^0_2:= \bar{\mathbf x}_2^0/\bignorm{\bar{\mathbf x}_2^0}$, and $\bar{\mathbf x}_1^0 := \truncatedvector{\mathcal A\mathbf x_2^0 \mathbf x_3^0   }{r_1}$, $\mathbf x_1^0 :=\bar{\mathbf x}_1^0/\bignorm{\bar{\mathbf x}_1^0}$. 
%	
%	
%	4. Return $(\mathbf x^0_1,\mathbf x^0_2,\mathbf x^0_3)$.
%\end{boxedminipage}
Algorithm \ref{proc:init2_order_d} combines the ideas of both Algorithms 1 and 2 of \cite{chan2016approximability} and extends them to higher-order tensors. When reducing to the matrix case, our algorithm here is still different from \cite{chan2016approximability}, as we find sparse singular vector pairs, while  \cite{chan2016approximability} pursues sparse eigenvectors of a symmetric matrix.

The computational complexity of Algorithm \ref{proc:init2_order_d}  is as follows. Computing the largest singular value of $\mathcal A\mathbf e_1^{i_1}\cdots\mathbf e_{d-2}^{i_{d-2}}$ takes $O(\min\{ n_{d-1}^2  n_d, n_{d-1}  n_d^2  \}   )$ flops in theory. Computing $\mathbf x^0_d$ takes $O(n_d\log(n_d))$ flops.  Thus the first two steps take $O(n_1\cdots n_{d-2}\min\{ n_{d-1}^2  n_d, n_{d-1}  n_d^2  \} + n_d\log(n_d))$ flops. The flops of the third step are the same as those of Algorithm \ref{proc:init1_order_d}.  As a result, the total complexity is dominated by    $O(n_1\cdots n_{d-2}\min\{ n_{d-1}^2  n_d, n_{d-1}  n_d^2  \}+ \sum^d_{j=1}n_j\log(n_j))$.

Algorithm \ref{proc:init2_order_d} is also well-defined as follow.
\begin{proposition}
	\label{prop:welldefined2} If $\mathcal A\neq 0$ and $(\mathbf x^0_1,\ldots,\mathbf x^0_d)$ is generated by Algorithm \ref{proc:init2_order_d}, then $\mathbf x^0_j\neq 0$, $1\leq j\leq d$.
\end{proposition}
\begin{proof}
	The definition of $(\bar i_1,\ldots, \bar i_{d-2})$ shows that the matrix $ \mathcal A\mathbf e_1^{\bar i_1}\cdots\mathbf e_{d-2}^{\bar i_{d-2}} \neq 0$, and hence $\bar{\mathbf x}_d\neq 0$ and  $\mathbf x^0_d\neq 0$.  We also observe from step 2 that $\bar{\mathbf x}_d  = \mathcal A\mathbf e_1^{\bar i_1}\cdots\mathbf e_{d-2}^{\bar i_{d-2}} \bar{\mathbf x}_{d-1}/\normalnorm{\mathcal A\mathbf e_1^{\bar i_1}\cdots\mathbf e_{d-2}^{\bar i_{d-2}} \bar{\mathbf x}_{d-1}}$, and so 
	\begin{eqnarray*}
		\innerprod{ \mathcal A\mathbf e_1^{\bar i_1}\cdots\mathbf e_{d-2}^{\bar i_{d-2}}\mathbf x^0_d   }{\bar{\mathbf x}_{d-1}} &=& \innerprod{\mathcal A\mathbf e_1^{\bar i_1}\cdots\mathbf e_{d-2}^{\bar i_{d-2}} \bar{\mathbf x}_{d-1}  }{\mathbf x^0_d} \\
		&=& \innerprod{\bar{\mathbf x}_d}{\mathbf x^0_d}\normalnorm{\mathcal A\mathbf e_1^{\bar i_1}\cdots\mathbf e_{d-2}^{\bar i_{d-2}} \bar{\mathbf x}_{d-1}} >0,
	\end{eqnarray*}
	implying that $ \mathcal A\mathbf e_1^{\bar i_1}\cdots\mathbf e_{d-2}^{\bar i_{d-2}}\mathbf x^0_d\neq 0$, and so $\mathbf x^0_{d-1} = \truncatedvector{ \mathcal A\mathbf e_1^{\bar i_1}\cdots\mathbf e_{d-2}^{\bar i_{d-2}}\mathbf x^0_d}{r_{d-1}}\neq 0$. 
	Similar arguments apply to show that $\mathbf x^0_{d-2} \neq 0,\ldots,\mathbf x^0_1\neq 0$ then. 
\end{proof}

To analyze the approximation bound, we need the following proposition.
\begin{proposition}
	\label{prop:lower_bound_prop_2}
	Let $\bar i_1,\ldots,\bar i_{d-2}$ and $(\bar{\mathbf x}_{d-1},\bar{\mathbf x}_d)$ be defined in Algorithm \ref{proc:init2_order_d}. Then it holds that
	\[
	\mathcal A\mathbf e_1^{\bar i_1}\cdots \mathbf e_{d-2}^{\bar i_{d-2}}\bar{\mathbf x}_{d-1}\bar{\mathbf x}_d \geq \max_{i_1,\ldots,i_{d-2},\normalnorm{\mathbf y}=\normalnorm{\mathbf z}=1,\bigzeronorm{\mathbf y}\leq r_{d-1},\bigzeronorm{\mathbf z}\leq r_d}\bigjueduizhi{\mathcal A\mathbf e_1^{i_1}\cdots \mathbf e_{d-2}^{i_{d-2}}\mathbf y\mathbf z}.
	\]
\end{proposition}
\begin{proof} The result holds by noticing the definition of $(\mathbf e^{\bar i_1}_1,\ldots,\mathbf e^{\bar i_{d-2}}_{d-2},\bar{\mathbf x}_{d-1},\bar{\mathbf x}_d)$,  and the additional sparsity constraints in the right-hand side of the inequality. 
\end{proof}

We first derive the approximation bound with $d=3$ as an illustration. 
\begin{theorem}
	\label{th:init_theory_bound_2}
	Let $v^{\rm opt}$ be defined as that in Theorem \ref{th:init_theory_bound} when $d=3$, and 	let $(\mathbf x^0_1,\mathbf x^0_2,\mathbf x^0_3)$ be generated by Algorithm \ref{proc:init2_order_d}. Then it holds that
	\[
	\mathcal A\mathbf x^0_1\mathbf x^0_2\mathbf x^0_3\geq \sqrt{\frac{r_2r_3}{n_2n_3r_1}}v^{\rm opt}.
	\]
\end{theorem}
\begin{proof}
	Denote $(\mathbf x^*_1,\mathbf x^*_2,\mathbf x^*_3)$ as a maximizer to \eqref{prob:str1approx_org_max}.	We have
	\begin{eqnarray}
		v^{\rm opt} &\leq& \sqrt{r_1}\max_{i_1}\bigjueduizhi{\mathcal A\mathbf e_1^{i_1}\mathbf x_2^*\mathbf x^*_3} \nonumber\\
		&\leq& \sqrt{r_1}\max_{i_1,\bignorm{\mathbf y}=\bignorm{\mathbf z}=1,\bigzeronorm{\mathbf y}\leq r_2,\bigzeronorm{\mathbf z}\leq r_3}\bigjueduizhi{\mathcal A\mathbf e_1^{i_1}\mathbf y\mathbf z}\nonumber\\
		&\leq&	 \sqrt{r_1}\mathcal A\mathbf e_1^{\bar i_1}\bar{\mathbf x}_2\bar{\mathbf x}_3,\label{eq:th:lower_bound_5}
	\end{eqnarray}
	where the first inequality is the same as \eqref{eq:th:lower_bound:1}, while Proposition \ref{prop:lower_bound_prop_2} gives the last one. Now denote $A:= \mathcal A\mathbf e_1^{\bar i_1}\in\mathbb R^{n_2\times n_3}$. Our remaining task is to show that
	\begin{equation}
		\label{eq:th:lower_bound_3}
		(\mathbf x^0_2)^\top A\mathbf x^0_3 \geq \sqrt{\frac{r_2r_3}{n_2n_3}}\bar{\mathbf x}_2^\top A\bar{\mathbf x}_3.
	\end{equation}
	Recalling the definition of $(\bar{\mathbf x}_2,\bar{\mathbf x}_3)$, we have $\lambda_{\max}(A) = \bar{\mathbf x}_2^\top A\bar{\mathbf x}_3$. Lemma \ref{lem:lower_bound_lem:1} tells us that
	\begin{equation}\label{eq:th:lower_bound_4}
		\bignorm{A\mathbf x^0_3}  \geq   \sqrt{\frac{r_3}{n_3}}\lambda_{\max}(A) =  \sqrt{\frac{r_3}{n_3}}\bar{\mathbf x}_2^\top A\bar{\mathbf x}_3;
	\end{equation}
	on the other hand, since
	\[
	\bar{\mathbf x}^0_2 = \truncatedvector{\mathcal A\mathbf e_1^{\bar i_1}\mathbf x^0_3}{r_2} = \truncatedvector{A\mathbf x^0_3}{r_2}, ~\mathbf x^0_2 = \bar{\mathbf x}_2^0/\bignorm{\bar{\mathbf x}_2^0},
	\]
	it follows   from Proposition \ref{prop:lower_bound_truncation_times_org} that
	\[
	\innerprod{A\mathbf x^0_3}{\mathbf x^0_2} \geq \sqrt{\frac{r_2}{n_2}}\bignorm{A\mathbf x^0_3},
	\]
	combining which with \eqref{eq:th:lower_bound_4}  gives \eqref{eq:th:lower_bound_3}. Finally, \eqref{eq:th:lower_bound_3} together with \eqref{eq:th:lower_bound_5} and   the definition of $\mathbf x^0_1$ yields 
	\begin{eqnarray*}
		\mathcal A\mathbf x^0_1\mathbf x^0_2\mathbf x^0_3 &\geq& \mathcal A\mathbf e_1^{\bar i_1}\mathbf x^0_2\mathbf x^0_3 =\innerprod{A\mathbf x^0_3}{\mathbf x^0_2} \geq \sqrt{\frac{r_2r_3}{n_2n_3}}\bar{\mathbf x}_2^\top A\bar{\mathbf x}_3 = \sqrt{\frac{r_2r_3}{n_2n_3}} \mathcal A\mathbf e_1^{\bar i_1}\bar{\mathbf x}_2\bar{\mathbf x}_3 \\
		& \geq& \sqrt{\frac{r_2r_3}{n_2n_3r_1}}v^{\rm opt},
	\end{eqnarray*}
	as desired.  
\end{proof}

The following approximation bound is presented for general order $d\geq 3$.
\begin{theorem}
	\label{th:init_theory_bound_2_general}
	Let $(\mathbf x^0_1,\ldots,\mathbf x^0_d)$ be generated by Algorithm \ref{proc:init2_order_d}. Then it holds that
	\[
	\mathcal A\mathbf x^0_1\cdots\mathbf x^0_d\geq \sqrt{\frac{r_{d-1}r_d}{n_{d-1}n_d \prod^{d-2}_{j=1}r_j  }}v^{\rm opt}.
	\]
\end{theorem}
\begin{proof}
	Similar to \eqref{eq:th:lower_bound_5}, we have
	\begin{eqnarray*}
		v^{\rm opt} &\leq& \sqrt{r_1}\max_{i_1}|\mathcal A\mathbf e_1^{i_1}\mathbf x_2^*\cdots\mathbf x^*_d| \leq \cdots \\
		&\leq& \sqrt{\prod^{d-2}_{j=1}\nolimits r_j} \max_{i_1,\ldots,i_{d-2}}\nolimits |\mathcal A\mathbf e_1^{i_1}\cdots\mathbf e_{d-2}^{i_{d-2}}\mathbf x^*_{d-1}\mathbf x^*_d   | \\ 
		&\leq&  \sqrt{\prod^{d-2}_{j=1}\nolimits r_j}\mathcal A\mathbf e_1^{\bar i_1}\cdots\mathbf e_{d-2}^{\bar i_{d-2}}\bar{\mathbf x}_{d-1}\bar{\mathbf x}_d.
	\end{eqnarray*}
	Similar to the proof of  \eqref{eq:th:lower_bound_3}, we can obtain
	\begin{eqnarray*}
		v^{\rm opt} &\leq& \sqrt{\prod^{d-2}_{j=1}\nolimits r_j}\mathcal A\mathbf e_1^{\bar i_1}\cdots\mathbf e_{d-2}^{\bar i_{d-2}}\bar{\mathbf x}_{d-1}\bar{\mathbf x}_d \\
		&\leq& \sqrt{\frac{n_{d-1}n_d \prod^{d-2}_{j=1}r_j }{r_{d-1}r_d  }}\mathcal A\mathbf e_1^{\bar i_1}\cdots\mathbf e_{d-2}^{\bar i_{d-2}}{\mathbf x}^0_{d-1}{\mathbf x}_d^0\\
		&\leq &
		\sqrt{\frac{ n_{d-1}n_d\prod^{d-2}_{j=1}r_j }{r_{d-1}r_d  }} \mathcal A\mathbf x^0_1\cdots\mathbf x^0_d.
	\end{eqnarray*}  
\end{proof}

\subsection{The third   algorithm}\label{sec:approx_alg3}
We begin with the illustration from third-order tensors. We will employ the Matlab function \texttt{reshape} to denote tensor folding/unfolding operations. For instance, given $\mathcal A\in\T$, $\mathbf a = \texttt{reshape}(\mathcal A, \prod^d_{j=1}n_j,1)$ means the unfolding of $\mathcal A$ to a vector $\mathbf a$ in $\mathbb R^{\prod^d_{j=1}n_j}$, while $\mathcal A = \texttt{reshape}(\mathbf a,n_1,\ldots, n_d)$ means the folding of   $\mathbf a$ back to   $\mathcal A$.  

\begin{boxedminipage}{0.92\textwidth}\small
	\begin{equation}  \label{proc:init3}
		\noindent {\rm Algorithm}~ (\mathbf x^0_1,\mathbf x^0_2,\mathbf x^0_3) =  {\rm approx\_alg}(\mathcal A,r_1,r_2,r_3) 
		\tag{C0}
	\end{equation}
	
	1. Unfold $\mathcal A$ to $A_1 := \texttt{reshape}(\mathcal A, n_1,n_2n_3)\in\mathbb R^{n_1\times n_2n_3}$; solve the matrix singular value problem 
	\[
	(\bar{\mathbf x}_1,\bar{\mathbf w}_1) \in \arg\max_{\bignorm{\mathbf x}=\bignorm{\mathbf w}=1}\nolimits \mathbf x^\top A_1\mathbf w;
	\]
	denote $\mathbf x^0_1 := \truncatedvector{\bar{\mathbf x}_1}{r_1}/\bignorm{\truncatedvector{\bar{\mathbf x}_1}{r_1}}$.
	
	2. Let $A_2 := \texttt{reshape}(A_1^\top \mathbf x^0_1, n_2,n_3)\in\mathbb R^{n_2\times n_3}$; solve the matrix singular value problem
	\[
	(\bar{\mathbf x}_2,\bar{\mathbf w}_2) \in\arg\max_{\bignorm{\mathbf y}=\bignorm{\mathbf w}=1}\nolimits \mathbf y^\top A_2\mathbf w;
	\]
	denote $\mathbf x^0_2:= \truncatedvector{\bar{\mathbf x}_2}{r_2}/\bignorm{\truncatedvector{\bar{\mathbf x}_2}{r_2}}$.

	3. Compute $\bar{\mathbf x}_3^0:=\truncatedvector{A_2^\top\mathbf x^0_2}{r_3}$, $\mathbf x^0_3:= \bar{\mathbf x}_3^0/\bignorm{\bar{\mathbf x}_3^0}$.

	4. Return $(\mathbf x^0_1,\mathbf x^0_2,\mathbf x^0_3)$.
\end{boxedminipage}

Different from Algorithm \ref{proc:init2_order_d}, Algorithm \ref{proc:init3} is mainly based on a series of computing leading singular vector pairs of certain matrices. In fact, Algorithm \ref{proc:init3} generalizes the approximation algorithm for dense tensor BR1Approx \cite[Algorithm 1 with DR 2]{he2010approximation} to our sparse setting; the main difference lies in the truncation of $\bar{\mathbf x}_j$ to obtain the sparse solution $\mathbf x^0_j$.  In particular, if no sparsity is required, i.e., $r_j=n_j,j=1,2,3$, then Algorithm \ref{proc:init3} boils down essentially to \cite[Algorithm 1 with DR 2]{he2010approximation}.  The next proposition shows that Algorithm \ref{proc:init3} is well-defined. 

\begin{proposition}
	\label{prop:welldefined30} If $\mathcal A\neq 0$  and $(\mathbf x^0_1,\ldots,\mathbf x^0_d)$ is generated by Algorithm \ref{proc:init3}, then $\mathbf x^0_j\neq 0$, $1\leq j\leq 3$.
\end{proposition}
\begin{proof} It is clear that $\bar{\mathbf x}_1\neq 0$, $\mathbf x^0_1\neq 0$, and $\bar{\mathbf w}_1\neq 0$ due to that $A_1\neq 0$. We have $\innerprod{\mathbf x^0_1}{A_1\bar{\mathbf w}_1} = \innerprod{\mathbf x^0_1}{\bar{\mathbf x}_1}\normalnorm{A_1\bar{\mathbf w}_1}>0$, and so $A_1^\top \mathbf x^0_1 \neq 0$, and $A_2\neq 0$. Similar argument shows that $\mathbf x^0_2\neq 0$ and $\mathbf x^0_3\neq 0$.
\end{proof}

Due to the presence of the truncation, deriving the approximation bound is different from that of \cite{he2010approximation}. In particular, we need Lemma \ref{lem:lower_bound_lem:1} to build bridges in the analysis. 
%To present the approximation bound,  
The following relation
\begin{equation}\label{eq:th:lower_bound:5}
	\mathcal A\mathbf x^0_1\mathbf x^0_2\mathbf x^0_3 = \innerprod{A^\top_1\mathbf x^0_1}{\mathbf x^0_3\otimes\mathbf x^0_2} = \innerprod{A_2}{\mathbf x^0_2(\mathbf x^0_3)^\top} = \innerprod{A^\top_2\mathbf x^0_2}{\mathbf x^0_3}
\end{equation}
is also helpful in the analysis, where $\otimes$ denotes the Kronecker product \cite{kolda2010tensor}.

\begin{theorem}
	\label{th:init_theory_bound_3}
	Let $v^{\rm opt}$ be defined as that in Theorem \ref{th:init_theory_bound} when $d=3$, and let $(\mathbf x^0_1,\mathbf x^0_2,\mathbf x^0_3)$ be generated by Algorithm \ref{proc:init3}. Then it holds that
	\begin{equation*}
		\mathcal A\mathbf x^0_1\mathbf x^0_2\mathbf x^0_3 \geq \sqrt{\frac{r_1r_2r_3}{n_1n_2n_3 n_2}} \lambda_{\max}(A_1) \geq \sqrt{\frac{r_1r_2r_3}{n_1n_2n_3 n_2}} v^{\rm opt}.
	\end{equation*}
\end{theorem}
\begin{proof}
	From the definition of $A_1$, $\bar{\mathbf x}_1$ and $\bar{\mathbf w}_1$,  we see that
	$$
	\lambda_{\max}(A_1) = \bar{\mathbf x}_1^\top A_1\bar{\mathbf w}_1 \geq \max_{\bignorm{\mathbf x}=\bignorm{\mathbf y}=\bignorm{\mathbf z}=1}\mathcal A\mathbf x\mathbf y\mathbf z \geq v^{\rm opt}.
	$$
	Therefore,	to prove the approximation bound, it suffices to show that
	\begin{equation}
		\label{eq:th:lower_bound:6}
		\mathcal A\mathbf x^0_1\mathbf x^0_2\mathbf x^0_3 \geq \sqrt{\frac{r_1r_2r_3}{n_1n_2n_3 n_2}} \lambda_{\max}(A_1).
	\end{equation}
	To this end, recalling Lemma \ref{lem:lower_bound_lem:1} and the definition of $\mathbf x^0_1$ (which is a truncation of the leading left singular vector of $A_1$), we obtain
	\begin{equation}
		\label{eq:th:lower_bound:7}
		\bignorm{A_1^\top\mathbf x^0_1} \geq \sqrt{\frac{r_1}{n_1}}\lambda_{\max}(A_1).
	\end{equation}
	Since $A_2 = \texttt{reshape}(A^\top_1\mathbf x^0_1,n_2,n_3)$, it holds that $\bigfnorm{A_2} = \bignorm{A^\top_1\mathbf x^0_1}$. Using again Lemma \ref{lem:lower_bound_lem:1} and recalling the definition of $\mathbf x^0_2$ (which is a truncation of the leading left singular vector of $A_2$), we get
	\begin{equation}
		\label{eq:th:lower_bound:8}
		\bignorm{A_2^\top \mathbf x^0_2} \geq \sqrt{\frac{r_2}{n_2}}\lambda_{\max}(A_2) \geq \sqrt{\frac{r_2}{n_2^2}}\bigfnorm{A_2} = \sqrt{\frac{r_2}{n_2^2}}\bignorm{A^\top_1\mathbf x^0_1},
	\end{equation}
	where the second inequality follows from the relation between the spectral norm and the Frobenius norm of a matrix. Finally, Proposition \ref{prop:lower_bound_truncation_times_org} and the definition of $\mathbf x^0_3$ gives that
	\begin{equation}
		\mathcal A\mathbf x^0_1\mathbf x^0_2\mathbf x^0_3=	\innerprod{A_2^\top\mathbf x^0_2}{\mathbf x^0_3} \geq \sqrt{\frac{r_3}{n_3}}\bignorm{A_2^\top \mathbf x^0_2},\label{eq:th:lower_bound:9}
	\end{equation}
	where the equality  follows from \eqref{eq:th:lower_bound:5}.	Combining  the above relation with \eqref{eq:th:lower_bound:7} and \eqref{eq:th:lower_bound:8} gives \eqref{eq:th:lower_bound:6}. This completes the proof.  
\end{proof}

When extending to $d$-th order tensors, the algorithm is presented as follows.

\begin{boxedminipage}{0.92\textwidth}\small
	\begin{equation}  \label{proc:init3_order_d}
		\noindent {\rm Algorithm}~ (\mathbf x^0_1,\ldots,\mathbf x^0_d) =  {\rm approx\_alg}(\mathcal A,\boldsymbol{r}) 
		\tag{C}
	\end{equation}
	
	1. Unfold $\mathcal A$ to $A_1 = \texttt{reshape}(\mathcal A, n_1,\prod^d_{j=2}n_j)\in\mathbb R^{n_1\times \prod^d_{j=2}n_j}$; solve the matrix singular value problem 
	\[
	(\bar{\mathbf x}_1,\bar{\mathbf w}_1) \in \arg\max_{\bignorm{\mathbf x_1}=\bignorm{\mathbf w_1}=1}\nolimits \mathbf x_1^\top A_1\mathbf w_1;
	\]
	denote $\mathbf x^0_1 := \truncatedvector{\bar{\mathbf x}_1}{r_1}/\bignorm{\truncatedvector{\bar{\mathbf x}_1}{r_1}}$.
	
	2. For $j=2,\ldots,d-1$, denote $A_j := \texttt{reshape}(A_{j-1}^\top\mathbf x_{j-1}^0,n_j,\prod^d_{k=j+1}n_k )\in\mathbb R^{n_j\times \prod^d_{k=j+1}n_k}$; solve the matrix singular value problem 
	\[
	(\bar{\mathbf x}_j,\bar{\mathbf w}_j) \in \arg\max_{\bignorm{\mathbf x_j}=\bignorm{\mathbf w_j}=1} \nolimits\mathbf x_j^\top A_j\mathbf w_j;
	\]
	denote $\mathbf x^0_j := \truncatedvector{\bar{\mathbf x}_j}{r_j}/\bignorm{\truncatedvector{\bar{\mathbf x}_j}{r_j}}$.
	
	3. Compute $\bar{\mathbf x}^0_d:=\truncatedvector{A_{d-1}^\top\mathbf x^0_{d-1}}{r_d}$, $\mathbf x^0_d:= \bar{\mathbf x}^0_d/\bignorm{\bar{\mathbf x}^0_d}$.
	
	4. Return $(\mathbf x^0_1,\ldots,\mathbf x^0_d)$.
\end{boxedminipage}
\begin{remark}\label{rmk:approx_3}
	In step 2, by noticing the recursive definition of $A_j$, one can check that $A_j$ is in fact the same as ${\texttt{\texttt{reshape}}}(\mathcal A\mathbf x_1^0\cdots\mathbf x_{j-1}^0,n_j,\prod^d_{k=j+1}n_k)$, where $\mathcal A\mathbf x^0_1\cdots \mathbf x^0_{j-1}$ is regarded as a tensor of size $n_{j}\times \cdots \times n_d$ with $\bigxiaokuohao{\mathcal A\mathbf x^0_1\cdots \mathbf x^0_{j-1}}_{i_j\cdots i_d} = \sum^{n_1,\ldots, n_{j-1}}_{i_1=1,\ldots,i_{j-1}=1} \mathcal A_{i_1\cdots i_{j-1} i_j\cdots i_d} (\mathbf x^0_1)_{i_1}\cdots (\mathbf x^0_{j-1})_{i_{j-1}}   $. 
\end{remark}

The computational complexity of the first step is $O(n_1^2 n_2\cdots n_d+n_1\log(n_1))$, where $O(n_1^2 n_2\cdots n_d)$ comes from solving the singular value problem of an $n_1\times \prod^d_{j=2}n_j$ matrix.     In the second step, for each $j$, computing $A^\top_{j-1}\mathbf x^0_{j-1} $ requires $O(n_{j-1}\cdots n_d)$ flops; computing the singular value problem requires $O(n^2_j n_{j+1}\cdots n_d )$ flops; thus
the complexity is $O(n_{j-1}\cdots n_d  + n^2_j n_{j+1}\cdots n_d + n_j\log(n_j))$. The third step is $O(n_{d-1}n_d + n_d\log(n_d))$. Thus the total complexity is dominated by $O(n_1\prod^d_{j=1}n_j + \sum^d_{j=1}n_j\log(n_j))$, which is the same as Algorithm \ref{proc:init2_order_d} in theory.  %We have the following approximation bound. 

%The computational complexity of Algorithm \ref{proc:init3_order_d} is  also $O(\prod^d_{j=1}n_j + \sum^d_{j=1}n_j\log(n_j))$.  

Similar to  Proposition \ref{prop:welldefined30}, we can show that Algorithm \ref{proc:init3_order_d} is well-defined, i.e., if $\mathcal A\neq 0$, then $\mathbf x^0_j\neq 0$, $1\leq j\leq d$. The proof is omitted. 

Concerning the approximation bound, similar to \eqref{eq:th:lower_bound:9}, one has  
$$\mathcal A\mathbf x^0_1\cdots\mathbf x^0_d = \innerprod{\mathcal A\mathbf x^0_1\cdots\mathbf x^0_{d-1}}{\mathbf x^0_d} =   \langle A^\top_{d-1}\mathbf x^0_{d-1},\mathbf x^0_d\rangle\geq \sqrt{\frac{r_d}{n_d}}\| A^\top_{d-1}\mathbf x^0_{d-1}\|,$$
where the second equality comes from  Remark \ref{rmk:approx_3} that $\mathcal A\mathbf x^0_1\cdots\mathbf x^0_{d-2} = A_{d-1}$ (up to a reshaping) and the inequality is due to Proposition \ref{prop:lower_bound_truncation_times_org}. Analogous to  \eqref{eq:th:lower_bound:8}, one can prove the following relation:
\[
\bignorm{A_j^\top\mathbf x^0_j}\geq \sqrt{\frac{r_j}{n_j^2}}\bignorm{A^\top_{j-1}\mathbf x^0_{j-1}},~2\leq j\leq d-1.
\]
Based on the above relations and \eqref{eq:th:lower_bound:7}, for order $d\geq 3$, we present the approximation bound without proof. 
\begin{theorem}
	\label{th:init_theory_bound_3_general}
	Let $(\mathbf x^0_1,\ldots,\mathbf x^0_d)$ be generated by Algorithm \ref{proc:init3_order_d}. Then it holds that
	\begin{small}
		\[
		\mathcal A\mathbf x^0_1\cdots\mathbf x^0_d\geq \sqrt{\frac{\prod^d_{j=1}r_j}{\prod^d_{j=1}n_j}}\frac{\lambda_{\max}(A_1)}{\sqrt{\prod^{d-1}_{j=2}n_j}} \geq \sqrt{\frac{\prod^d_{j=1}r_j}{\prod^d_{j=1}n_j}}\frac{v^{\rm opt}}{\sqrt{\prod^{d-1}_{j=2}n_j}} .
		\]
	\end{small}
\end{theorem}

\subsection{The fourth algorithm}
Algorithm \ref{proc:init3_order_d} computes $\bar{\mathbf x}_j$ from $A_j$  via solving singular value problems. When the size of the tensor is huge, this might be time-consuming. To further accelerate the algorithm, we propose the following algorithm, 
which is similar to Algorithm \ref{proc:init3_order_d}, while it  obtains $\bar{\mathbf x}_j$ without solving singular value problems. Denote $A^k$ as the $k$-th row of a matrix $A$. The algorithm   is presented as follows:

\begin{boxedminipage}{0.92\textwidth}\small
	\begin{equation}  \label{proc:init4_order_d}
		\noindent {\rm Algorithm}~ (\mathbf x^0_1,\ldots,\mathbf x^0_d) =  {\rm approx\_alg}(\mathcal A,\boldsymbol{r}) 
		\tag{D}
	\end{equation}
	
	1. Unfold $\mathcal A$ to $A_1 = \texttt{reshape}(\mathcal A, n_1,\prod^d_{j=2}n_j)$; let $A_1^{\tilde k}$ be the row of $A_1$ with the largest magnitude, i.e., $\normalnorm{ A_1^{\tilde k}}=\max_{1\leq k\leq n_1}\bigfnorm{A_1^k}$. Denote $\mathbf w_1:= (A_1^{\tilde k})^\top/\normalnorm{(A_1^{\tilde k})^\top}$ and let
	\[
	\bar{\mathbf x}_1  =  A_1\mathbf w_1;
	\]
	denote $\mathbf x^0_1 := \truncatedvector{\bar{\mathbf x}_1}{r_1}/\bignorm{\truncatedvector{\bar{\mathbf x}_1}{r_1}}$.
	
	2. For $j=2,\ldots,d-1$, denote $A_j := \texttt{reshape}(A_{j-1}^\top\mathbf x_{j-1}^0,n_j,\prod^d_{k=j+1}n_k )$; let $A_j^{\tilde k}$ be the row of $A_j$ with the largest magnitude. Denote $\mathbf w_j:= (A_j^{\tilde k})^\top/\normalnorm{(A_j^{\tilde k})^\top}$ and let
	\[
	\bar{\mathbf x}_j  =  A_j\mathbf w_j;
	\]
	denote $\mathbf x^0_j := \truncatedvector{\bar{\mathbf x}_j}{r_j}/\bignorm{\truncatedvector{\bar{\mathbf x}_j}{r_j}}$.
	
	3. Compute $\bar{\mathbf x}^0_d:=\truncatedvector{A_{d-1}^\top\mathbf x^0_{d-1}}{r_d}$, $\mathbf x^0_d:= \bar{\mathbf x}^0_d/\bignorm{\bar{\mathbf x}^0_d}$.
	
	4. Return $(\mathbf x^0_1,\ldots,\mathbf x^0_d)$.
\end{boxedminipage}

It is clear from the above algorithm that  computing $\bar{\mathbf x}_j$ only requires some matrix-vector productions, which has lower computational complexity than computing singular vectors. Numerical results presented in the next section will show the efficiency and effectiveness of this simple  modification. 

In Algorithm \ref{proc:init4_order_d}, the first step needs $O(n_1\cdots n_d + n_1\log(n_1))$ flops; in the second step, for each $j$, the complexity is $O(n_{j-1}\cdots n_d + n_j\log(n_j))$; the third step is $O(n_{d-1}n_d + n_d\log(n_d))$. Thus the total complexity is dominated by $O(n_1\cdots n_d + \sum^d_{j=1}n_j\log(n_j))$, which is lower than that of Algorithm \ref{proc:init3_order_d}, due to the SVD-free computation of $\bar{\mathbf x}_j$'s.

Reducing to the dense tensor setting, i.e.,  $r_j=n_j$ for each $j$, Algorithm \ref{proc:init4_order_d} is even new for dense tensor BR1Approx problems; when $d=2$,   similar ideas have not been   applied to approximation algorithms for  sparse matrix PCA/SVD yet.  

The next proposition shows that Algorithm \ref{proc:init4_order_d} is well-defined, whose proof is quite similar to that of Proposition \ref{prop:welldefined30} and is omitted.

\begin{proposition}
	\label{prop:welldefined4} If $\mathcal A\neq 0$ and $(\mathbf x^0_1,\ldots,\mathbf x^0_d)$ is generated by Algorithm \ref{proc:init4_order_d}, then $\mathbf x^0_j\neq 0$, $1\leq j\leq d$.
\end{proposition}

%The proof of the well-definedness of Algorithm \ref{proc:init4_order_d} is quite similar to that of Proposition \ref{prop:welldefined30}, and we omit it here. 

The approximation bound analysis  essentially relies on the following lemma. 
\begin{lemma}
	\label{lem:proc4}
	Given $A\in\mathbb R^{m\times n}$, with $A^{\tilde k}$ being the row of $A$ having the largest magnitude. Let $\mathbf w=(A^{\tilde k})^\top/\normalnorm{(A^{\tilde k})^\top}$, $\mathbf x = A\mathbf w$, and $\mathbf x^0 = \truncatedvector{\mathbf x}{r}/\normalnorm{\truncatedvector{\mathbf x}{r}}$ with $1\leq r\leq m$. Then there holds 
	\[
	\bignorm{A^\top\mathbf x^0} \geq \sqrt{\frac{r}{m^2} }\bigfnorm{A}.
	\]
\end{lemma}
\begin{proof}
	We have
	\begin{eqnarray*}
		\bignorm{A^\top \mathbf x^0} &=& \max_{\normalnorm{\mathbf z}=1 }\nolimits \innerprod{A^\top\mathbf x^0 }{\mathbf z} \\
		&\geq& \innerprod{A^\top \mathbf x^0}{\mathbf w} = \innerprod{\mathbf x}{\mathbf x^0} \\
		&\geq& \sqrt{\frac{r}{m}}\normalnorm{\mathbf x}, 
	\end{eqnarray*}
	where the second inequality follows from that $\mathbf w$ is normalized, and the   last one comes from Proposition \ref{prop:lower_bound_truncation_times_org}. We also have from the definition of $\mathbf x$ that
	\begin{eqnarray*}
		\normalnorm{\mathbf x}^2 &=& \sum^m_{k=1}\nolimits  ({A^k\mathbf w})^2 \geq  ({A^{\tilde k}\mathbf w})^2 \\
		&=& \innerprod{A^{\tilde k}}{\frac{A^{\tilde k}}{\normalnorm{A^{\tilde k}}}}^2  = \bignorm{A^{\tilde k}}^2 \\
		&\geq& \frac{1}{m}\bigfnorm{A}^2,
	\end{eqnarray*}
	where the last inequality comes from the definition of $A^{\tilde k}$. Combining the above analysis gives the desired result.  
\end{proof}
\begin{theorem}
	\label{th:init_theory_bound_4_general}
	Let $(\mathbf x^0_1,\ldots,\mathbf x^0_d)$ be generated by Algorithm \ref{proc:init4_order_d}. Then it holds that
	\begin{small}
		\[
		\mathcal A\mathbf x^0_1\cdots\mathbf x^0_d\geq \sqrt{\frac{\prod^d_{j=1}r_j}{\prod^d_{j=1}n_j}}\frac{\bigfnorm{\mathcal A}}{\sqrt{\prod^{d-1}_{j=1}n_j}}  \geq  \sqrt{\frac{\prod^d_{j=1}r_j}{\prod^d_{j=1}n_j}}\frac{v^{\rm opt}}{\sqrt{\prod^{d-1}_{j=1}n_j}} .
		\]
	\end{small}
\end{theorem}
\begin{proof}
	As the discussions above Theorem \ref{th:init_theory_bound_3_general} we get $\mathcal A\mathbf x^0_1\cdots\mathbf x^0_d \geq \sqrt{ \frac{r_d}{n_d}  }\bignorm{A^\top_{d-1} \mathbf x^0_{d-1} }$. Using Lemma \ref{lem:proc4}, for $2\leq j\leq d-1$, we have
	\[
	\bignorm{A^\top_{j}\mathbf x^0_{j}} \geq \sqrt{\frac{r_{j}}{n_{j}^2}}\bigfnorm{ A_{j} } = \sqrt{\frac{r_{j}}{n_{j}^2}} \bignorm{A^\top_{j-1}\mathbf x^0_{j-1}}. 
	\]
	In particular, from step 1 of the algorithm, $\bignorm{A^\top_1\mathbf x^0_1} \geq \sqrt{\frac{r_1}{n_1^2}}\bigfnorm{A_1} = \sqrt{\frac{r_1}{n_1^2}}\bigfnorm{\mathcal A}$. It is clear that $\bigfnorm{\mathcal A} \geq \lambda_{\max}(A_1) \geq v^{\rm opt}$. Combining the analysis gives the desired results.   
\end{proof}

Before ending this section, we summarize the approximation ratio and computational complexity of the proposed algorithms in Table \ref{tab:approx_alg_summary}. For convenience we set $r_1=\cdots=r_d$ and $n_1=\cdots = n_d$.

\begin{table}[htbp]
	\centering
	\caption{\footnotesize   Comparisons of the proposed approximation algorithms   on the approximation bound and computational complexity.}
	\begin{mytabular1}{ccc}
		\toprule
		\multicolumn{1}{c}{Algorithm} & Approximation bound &  Computational complexity  \\
		\toprule
		Algorithm	\ref{proc:init1_order_d} 	& $\frac{v^{\rm opt}}{\sqrt{ r^{d-1} }}$   &	$O(n^d\log(n) + dn\log(n) )$ \\
		\midrule
		Algorithm	\ref{proc:init2_order_d} 	& $\sqrt{\frac{r^2}{n^2}}\frac{v^{\rm opt}}{\sqrt{r^{d-2}}}$ &	$O(n^{d+1} + dn\log(n))$ \\
		\midrule
		Algorithm	\ref{proc:init3_order_d} 	& $\sqrt{\frac{r^d}{n^d}} \frac{\lambda_{\max}(A_1)}{\sqrt{n^{d-2}}}$ &	$O(n^{d+1} + dn\log(n))$ \\
		\midrule
		Algorithm	\ref{proc:init4_order_d} 	& $\sqrt{\frac{r^d}{n^d}} \frac{ \bigfnorm{\mathcal A}  }{\sqrt{n^{d-1}}}$ &	$O(n^{d} + dn\log(n))$ \\
		\bottomrule
	\end{mytabular1}%
	\label{tab:approx_alg_summary}%
\end{table}%

Concerning the approximation ratio $\mathcal A\mathbf x^0_1\cdots\mathbf x^0_d/v^{\rm opt}$, we see that if $r$ is a constant, then  the ratio of Algorithm \ref{proc:init1_order_d} is also a constant, while those of other algorithms  rely on $n$. This is the advantage of Algorithm \ref{proc:init1_order_d}, compared with other algorithms. 
When $d=2$, the ratios of Algorithms \ref{proc:init1_order_d} and \ref{proc:init2_order_d} are respectively $1/\sqrt{r}$ and $r/n$, which coincide with their matrix counterparts; see, \cite[Algorithms 1 and 2]{chan2016approximability}. We also observe that Algorithm \ref{proc:init2_order_d} generalizes the approximation algorithm and bound for dense tensor BR1Approx in \cite[Theorem 5.1]{zhang2012cubic}: If sparsity is not required, i.e., $r=n$, the bound of Algorithm \ref{proc:init2_order_d}   boils down to those of  \cite[Theorem 5.1]{zhang2012cubic}. 
For Algorithm \ref{proc:init3_order_d},  when $r=n$, or $r$ is proportional to $n$ up to a constant,   the ratio   reduces to $ O(1/\sqrt{ n^{d-2}})$, which recovers that of \cite[Algorithm 1]{he2010approximation}. %Comparing with the algorithms and  analysis in \cite{he2010approximation,chan2016approximability}, we can see that the proposed algorithms are non-trivial extensions of \cite{he2010approximation,chan2016approximability}. 
Although the ratio of Algorithm \ref{proc:init4_order_d} is worse than that of Algorithm \ref{proc:init3_order_d} with an additional factor $1/\sqrt{n}$, we shall also observe that the numerator of the bound of Algorithm \ref{proc:init4_order_d} is $\bigfnorm{\mathcal A}$, while that of Algorithm \ref{proc:init3_order_d} is $\lambda_{\max}(A_1)$ (see Algorithm \ref{proc:init3_order_d} for the definition of $A_1$), where the latter is usually much smaller than the former. 
Note that two randomized approximation algorithms as initializations were proposed in  \cite[Algorithms 3 and 4]{sun2017provable} (see its arXiv version), where the performance analysis was considered on structured tensors. It would be also interesting to study their approximation bound for general tensors. Concerning the computational complexity, we see that Algorithm \ref{proc:init4_order_d} admits the lowest one in theory.  
This is also confirmed by our numerical observations that will be presented in the next section. Note that if the involved singular value problems in Algorithm \ref{proc:init3_order_d} are solved by the Lanczos algorithm at $k$ steps \cite{comon1990tracking} where $k$ is a user-defined parameter, then the computational complexity is $O(k n^d + dn\log(n))$. 

\color{black}
Finally,  We discuss that whether the approximation bounds are achieved. We only consider the cases that $r_j<n_j$. We first consider Algorithm \ref{proc:init3_order_d} and take $d=3$ as an example.  
In  \eqref{eq:th:lower_bound:8}, achieving $\lambda_{\max}(A_2)=\frac{\|A_2\|_F}{\sqrt{n_2}}$ requires that ${\rm rank}(A_2)=n_2$ (assuming that $n_2\leq n_3$) and all the singular values are the same. If these are true, then $\| A_2^{\top}\mathbf x_2^0\| > \sqrt{ \frac{r_2}{n_2} }\lambda_{\max}(A_2)$.\footnote{Otherwise, from the analysis of Lemma \ref{lem:lower_bound_lem:1},  $\| A_2^{\top} {\mathbf x_2^0}\| = \sqrt{  \frac{r_2}{n_2} }\lambda_{\max}(A_2)$ if and only if  1) $A_2^{\top}\bar{\mathbf x}_2 = \alpha A_2^{\top}\mathbf x_2^0$ for some $\alpha \in\mathbb R$,  and 2) $\innerprod{\bar{\mathbf x}_2}{\mathbf x_2^0} = \sqrt{ \frac{r_2}{n_2}  }$. 2) holds if and only if every entry of $\bar{\mathbf x}_2$ takes the same value, which together with $r_2<n_2$ implies that $\bar{\mathbf x}_2-\alpha \mathbf x_2^0\neq 0$; however, this and 1) lead to ${\rm rank}(A_2)<n_2$, deducing a contradiction.} Thus the bound of Algorithm \ref{proc:init3_order_d} may not be tight. 
%achieving the approximation ratios $\sqrt{\frac{r_1}{n_1}},\sqrt{ \frac{r_2}{n_2}}, \sqrt{\frac{r_3}{n_3}}$ during the analysis respectively requires that $A_1$ in \eqref{eq:th:lower_bound:7}, $A_2$ in \eqref{eq:th:lower_bound:8}, and $A_2\mathbf x^0_2$ in \eqref{eq:th:lower_bound:9} are all-one matrices or vectors. These are in fact reasonable when $\mathcal A$ is an all-one tensor. However, if this is the case, then the ratio $\lambda_{\max}(A_2)\geq \|A_2\|_F/\sqrt{n_2}$ in \eqref{eq:th:lower_bound:8} is not tight (if $\mathcal A$ is an all-one tensor, then  $A_2$ satisfies $\lambda_{\max}(A_2) = \|A_2\|_F$).  Thus the ratio $\sqrt{\frac{r_1r_2r_3}{n_1n_2^2n_3}}$ might not be tight.  %There exist two special cases that the bound of Algorithm \ref{proc:init3_order_d} can be tight. 
The approximation ratio of Algorithm \ref{proc:init4_order_d} relies on Lemma \ref{lem:proc4}. However, there do not exist matrices achieving the approximation ratio $\sqrt{\frac{r}{m^2}}$ in Lemma \ref{lem:proc4}, implying that the   bound of Algorithm \ref{proc:init4_order_d} may not be tight. 
%The case of Algorithm \ref{proc:init2_order_d} is similar to Algorithm \ref{proc:init3_order_d}: Achieving the ratios $\sqrt{\frac{r_2}{n_2}}$ and $\sqrt{\frac{r_3}{n_3}}$ in the analysis of Algorithm \ref{proc:init2_order_d} requires the all-one tensor $\mathcal A$, while if this is the case, then the remaining ratio $1/\sqrt{r_1}$ cannot be achieved. 
For Algorithm \ref{proc:init1_order_d}, if step 3 is not executed (use $\mathbf e_1^{\bar i_1}, \ldots, \mathbf e_{d-1}^{\bar i_{d-1}}, \mathbf x_d^0$ obtained in step 2 as the output), then the bound is tight: Consider $\mathcal A\in \mathbb R^{n_1\times n_2\times n_3}$ with $n_1=n_2=n_3=n$ as an all-one tensor and let $1< r_1=r_2=r_3 = r<n$; then $v^{\rm opt} = r^{\frac{3}{2}}$. The generated feasible point is $\mathbf x^0_j = \mathbf e^1_j$, $j=1,2$, and $x^0_3$ is the vector whose first $r$ entries are $\frac{1}{\sqrt r}$ and the remaining ones are zero.   Then $\mathcal A\mathbf x^0_1\mathbf x^0_2\mathbf x^0_3/v^{\rm opt} = \frac{1}{r}$, achieving the bound.  For Algorithm \ref{proc:init2_order_d}, in case that $d=3$ and if $\mathbf x_1^0$ is not updated by step 3  (use $\mathbf e_1^{\bar i_1}$ obtained in step 2 as the output), then the bound can be tight: Consider $\mathcal A\in \mathbb R^{4\times 4\times 4}$, with $\mathcal A(i,:,:) = A:= \left[\begin{smallmatrix}
	0&1&0&1\\
	0&1&0&1\\
	1& 0&1&0\\
	1& 0&1&0\\
\end{smallmatrix}\right]$, $i=1,\ldots,4$, and let $r_1=r_2=r_3=2$. Then $v^{\rm opt}= 2\sqrt 2$.\footnote{Due to the structure of $\mathcal A$, $\mathcal A\mathbf x_1\mathbf x_2\mathbf x_3 =  \innerprod{\mathbf e}{\mathbf x_1}\cdot\mathbf x_2^{\top}A\mathbf x_3$ with $\mathbf e = [1~1~1~1]^{\top}$. It is clear that $\mathbf x_1^*=\frac{1}{\sqrt 2}[ 1~ 1~ 0~ 0]^{\top}$. Denote $\mathbf x_2^*:=\mathbf x_1^*$, and $\mathbf x_3^*:=\frac{1}{\sqrt 2}[0~1~0~1]^{\top}$. Since $(\mathbf x_2^*)^{\top}A\mathbf x_3^* = \lambda_{\max}(A) = 2$, $(\mathbf x_2^*,\mathbf x_3^*)$ is a maximizer to $\max_{\|\mathbf x_i\|=1,\|\mathbf x_i\|_0\leq 2,i=2,3}\mathbf x_2^{\top}A\mathbf x_3$, and so $(\mathbf x_1^*,\mathbf x_2^*,\mathbf x_3^*)$ is a maximizer to \eqref{prob:str1approx_org_max}, with $\mathcal A\mathbf x_1^*\mathbf x_2^*\mathbf x_3^*=2\sqrt 2$.} Applying Algorithm \ref{proc:init2_order_d} to $\mathcal A$ yields $\bar i_1=1$, with $\bar{\mathbf x}_3=[1~1~1~1]^{\top}/2$;  then $\mathbf x_3^0= [1~1~0~0]^{\top}/\sqrt 2$, and $\mathbf x_2^0=\mathbf x_3^0$. Finally,    $ {\mathcal A\mathbf e_1^1\mathbf x_2^0\mathbf x_3^0 }/{v^{\rm opt}}= \frac{1}{2\sqrt 2} = \sqrt{\frac{r^2}{n^2 r}}$. 
In summary, the reason that why the bounds cannot be achieved is that although we derive the worst-case inequalities in every step of the analysis, after putting them together, the final approximation bounds might not be tight. How to improve them still need further research.

\color{black}

%Point 2 of Theorem \ref{th:convergence} means that if the initial point is generated by the approximation algorithms in Sect. \ref{sec:approx_alg}, then the sparsity will not be changed after running Algorithm \ref{alg:reweighted}. 
\section{Numerical Experiments}\label{sec:numer_experiments}
We evaluate the proposed algorithms in this section on synthetic and real data.  
All the   computations are conducted on an Intel i7  CPU desktop computer with 16 GB of RAM. The supporting software is Matlab R2019a.   Tensorlab  \cite{tensorlab2013} is employed for basic tensor operations.   

\paragraph{Performance of approximation algorithms}
We first compare Algorithms \ref{proc:init1_order_d}, \ref{proc:init2_order_d},   \ref{proc:init3_order_d}, and \ref{proc:init4_order_d} on solving \eqref{prob:str1approx_org_max}. The tensor is given by 
\begin{equation}\label{exp:tr1a}
	\mathcal A = \sum^R_{i=1}\nolimits\mathbf x_{1,i}\circ\cdots\circ\mathbf x_{d,i} \in\T,
\end{equation}
where $\mathbf x_{j,i}\in\mathbb R^{n_j}$, $i=1,\ldots,R$, $j=1,\ldots,d$, and we let $R=10$. Here the vectors are first randomly drawn from the normal distribution, and then $sr = 70\%$ of the entries are randomly set to be zero.  We set $d=3,4$, and  let $n_1=\cdots=n_d=n$ with $n$ varying  from $5$ to $100$. For each case, we randomly generated $50$ instances, and the   averaged results are presented. $r_j = \lfloor (1-sr)n_j\rfloor$ for each $j$ in \eqref{prob:str1approx_org_max}. As a baseline, we also randomly generate feasible points $(\mathbf x^{\rm random}_1,\ldots,\mathbf x^{\rm random}_d)$ and evaluate its performance. On the other hand, we easily see that
\begin{equation}\label{eq:upper_bound}
	v^{\rm ub} := \min\{\lambda_{\max}(A_{(1)}),\ldots,\lambda_{\max}(A_{(d)})  \}
\end{equation}
is an upper bound for problem \eqref{prob:str1approx_org_max}, where $A_{(j)} = \texttt{reshape}(\mathcal A,n_j,\prod^d_{k\neq j}n_k)$ denotes the $j$-mode unfolding of $\mathcal A$. Thus we also evaluate $v^{\rm ub}$. 
The results are depicted in Fig. \ref{fig:tr1a}, where the left panels show the curves of the objective values $\mathcal A\mathbf x^{0}_1\cdots\mathbf x^0_d$ versus $n$ of different algorithms, whose colors are respectively   cyan (Alg. \ref{proc:init1_order_d}), magenta (Alg. \ref{proc:init2_order_d}),   blue (Alg. \ref{proc:init3_order_d}), and red (Alg. \ref{proc:init4_order_d}); the curves of the random value $\mathcal A\mathbf x^{\rm random}_1\cdots\mathbf x^{\rm random}_d$ is  in   black with hexagram markers, while   the curve of the upper bounds $v^{\rm ub}$ is in             black with diamond markers. The right ones plot the curve of CPU time versus $n$.

\begin{figure}[h] 
	\centering
	\subfigure[Objective value and CPU time when $d=3$.]{
		\label{fig:d3_approx} %% label for second subfigure
		\includegraphics[height=3.5cm,width=13cm]{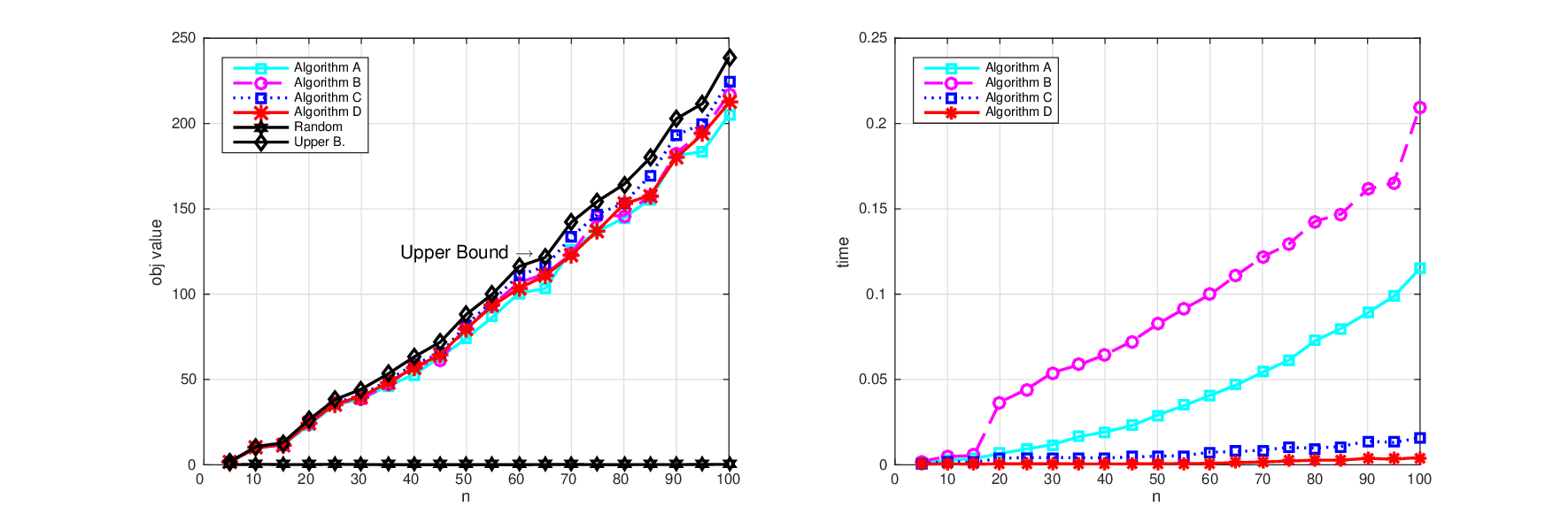}}\\
	\subfigure[Objective value and CPU time  when $d=4$.]{
		\label{fig:d4_approx}%% label for second subfigure
		\includegraphics[height=3.5cm,width=13cm]{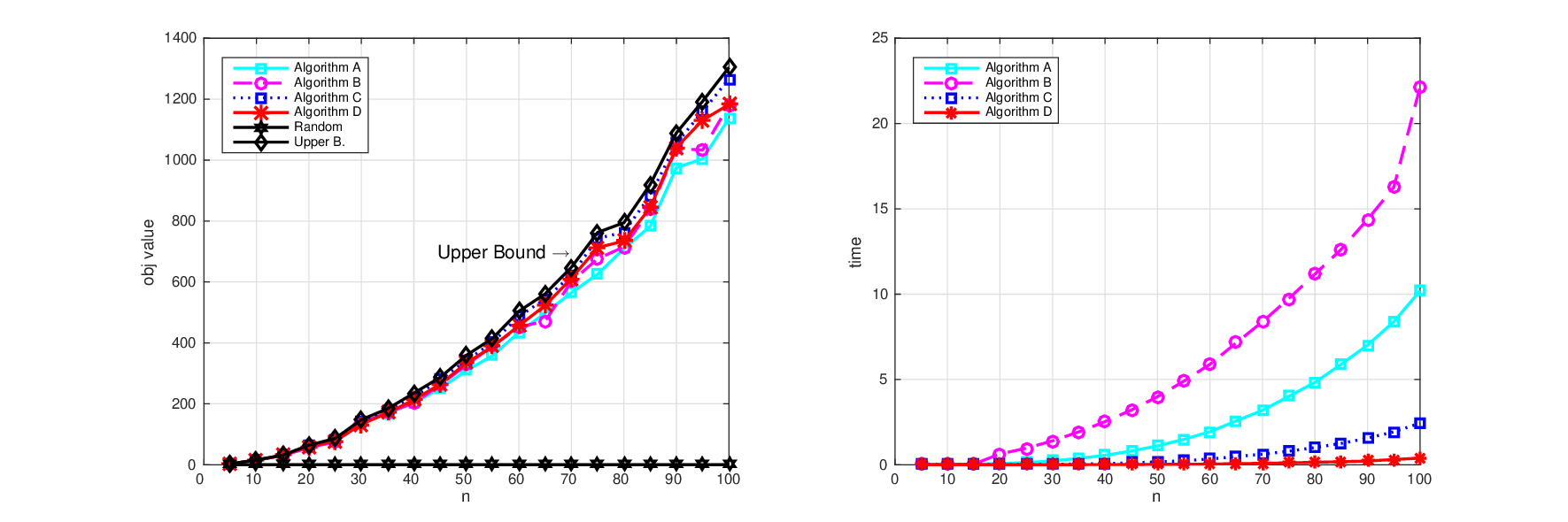}}\\
	%    \hspace{0.6in}
	%	\subfigure[A $5\times 40\times 40\times 40$ tensor with $R=20$ and $t=3$.]{
		%		\label{fig:compare_dim5_40_40_40_r20}%% label for second subfigure
		%		\includegraphics[height=3.5cm,width=14cm]{fig/3.eps}}    
	\caption{Comparisons of Algorithms \ref{proc:init1_order_d}, \ref{proc:init2_order_d},   \ref{proc:init3_order_d}, and \ref{proc:init4_order_d} for solving \eqref{prob:str1approx_org_max} where $\mathcal A$ is given by \eqref{exp:tr1a}. $n$ varies from $5$ to $100$. Left panels: objective value $\mathcal A\mathbf x^0_1\cdots\mathbf x^0_d$ versus $n$; right panels: CPU time.} 
	
	\label{fig:tr1a} %% label for entire figure
\end{figure}

From the left panels, we observe that the objective values  generated by Algorithms \ref{proc:init1_order_d}, \ref{proc:init2_order_d},   \ref{proc:init3_order_d}, and \ref{proc:init4_order_d} are similar, where Algorithm \ref{proc:init3_order_d} performs better; Algorithm \ref{proc:init4_order_d} performs the second when $d=4$, and it is comparable with Algorithm \ref{proc:init2_order_d} when $d=3$; Algorithm \ref{proc:init1_order_d} gives the worst results, which may be that Algorithm \ref{proc:init1_order_d} does not explore the structure of the problem as much as possible. We also observe that the objective values of all the   algorithms are quite close to the upper bound \eqref{eq:upper_bound}, which demonstrates the effectiveness of the proposed algorithms. In fact, the ratio of $\frac{\mathcal A\mathbf x^0_1\cdots\mathbf x^0_d}{v^{\rm ub}}$ is in $(0.7,1)$, which is far better than the %worst-case 
approximation ratios presented in Sect. \ref{sec:approx_alg}. This implies that at least for this kind of tensors, the approximation ratios might be independent of the size of the tensor.
The   value   $\mathcal A\mathbf x^{\rm random}_1\cdots\mathbf x^{\rm random}_d$ is close to zero (the curve almost coincides with the $x$-axis).  Concerning the computational time, Algorithms \ref{proc:init4_order_d} is the most efficient one, confirming the theoretical results in Table \ref{tab:approx_alg_summary}.   Algorithm \ref{proc:init3_order_d} is the second efficient one. Algorithms \ref{proc:init1_order_d} and \ref{proc:init2_order_d} do not perform well compared with Algorithms \ref{proc:init3_order_d} and  \ref{proc:init4_order_d}, although their computational complexity is similar in theory. The reason may be because the first two algorithms require for-loop operations, which is  time-consuming in Matlab. 

\begin{figure}[h] 
	\centering
	\subfigure[Objective value and CPU time when $d=3$.]{
		\label{fig:d3_approx_sparsity} %% label for second subfigure
		\includegraphics[height=3.5cm,width=13cm]{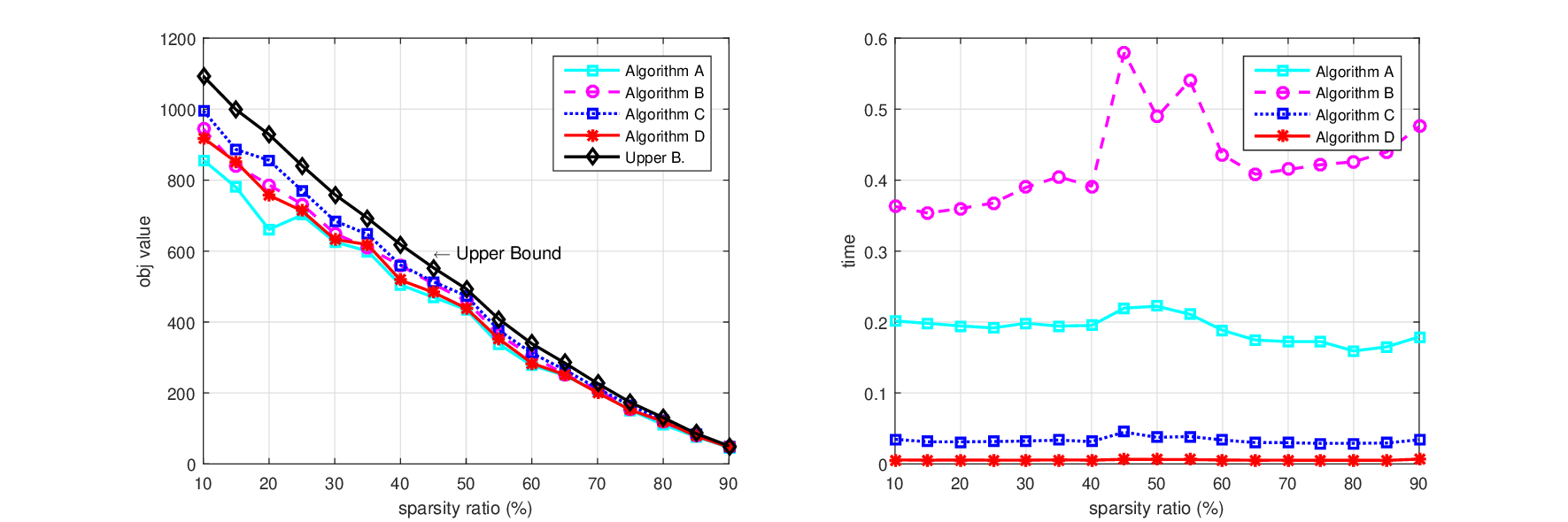}}\\
	\subfigure[Objective value and CPU time when $d=4$.]{
		\label{fig:d4_approx_sparsity}%% label for second subfigure
		\includegraphics[height=3.5cm,width=13cm]{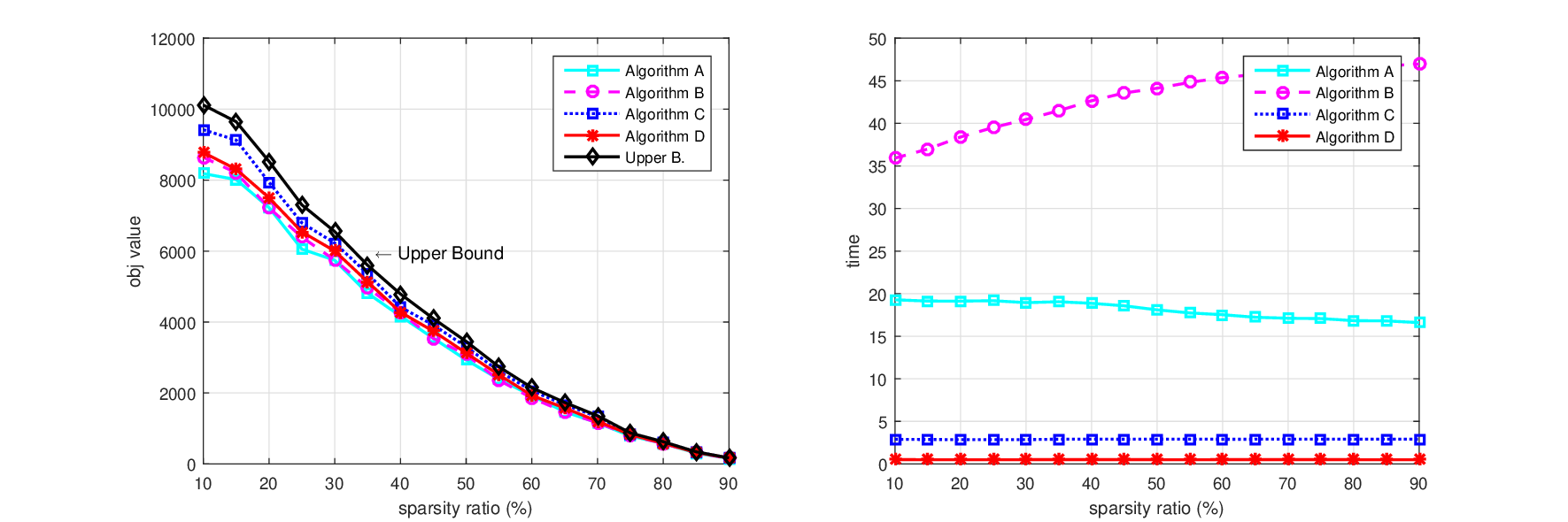}}\\
	\caption{Comparisons of Algorithms \ref{proc:init1_order_d}, \ref{proc:init2_order_d},   \ref{proc:init3_order_d}, and \ref{proc:init4_order_d} for solving \eqref{prob:str1approx_org_max} with fixed $n$ and   sparsity ratio  varying    from $10\%$ to $90\%$. Fig. \ref{fig:d3_approx_sparsity}: $\mathcal A\in\mathbb R^{100\times 100\times 100}$; Fig. \ref{fig:d4_approx_sparsity}: $\mathcal A\in\mathbb R^{100\times 100\times 100\times 100}$.   Left panels: the objective values $\mathcal A\mathbf x^0_1\cdots\mathbf x^0_d$ versus sparsity ratio; right panels: CPU time.} 
	
	\label{fig:sparsity} %% label for entire figure
\end{figure}

We then consider fix $n=100$ and vary the sparsity ratio $sr$ from $10\%$ to $90\%$ of $\mathcal A$ in \eqref{exp:tr1a}, and compare the performance of the four proposed algorithms. $r_j$ in \eqref{prob:str1approx_org_max} is   set  to $\lfloor (1-sr) n_j\rfloor$ correspondingly. The results of the objective values $\mathcal A\mathbf x^0_1\cdots \mathbf x^0_d$    together with the upper bound are depicted in the left panels of Fig. \ref{fig:sparsity}, from which we still observe that all the algorithms are close to the upper bound \eqref{eq:upper_bound}; among them, Algorithm \ref{proc:init3_order_d} is still slightly better than the other three, followed by Algorithms \ref{proc:init2_order_d} and \ref{proc:init4_order_d}. The CPU time is plotted in the right panels of Fig. \ref{fig:sparsity}, which still shows that Algorithm \ref{proc:init4_order_d} is the most efficient one. In fact, in average, Algorithm \ref{proc:init4_order_d} is $5$ times faster than Algorithm \ref{proc:init3_order_d}, which ranks the  second, and is about $80$ times faster than Algorithm \ref{proc:init2_order_d}, which is the slowest one.

\color{black}

Overall, comparing with Algorithms \ref{proc:init1_order_d} and \ref{proc:init2_order_d} that find the solutions fiber by fiber, or slide by slide, Algorithm \ref{proc:init3_order_d}  admits hierarchical structures that take the whole tensor into account, and so it   can explore the structure of the data tensor better. 
This may be the reason why Algorithm \ref{proc:init3_order_d} is better than Algorithms \ref{proc:init1_order_d} and \ref{proc:init2_order_d} in terms of the effectiveness. When compared with Algorithm \ref{proc:init4_order_d}, Algorithm \ref{proc:init3_order_d} computes each $\mathbf x^0_j$ ``optimally'' via SVD, while Algorithm \ref{proc:init4_order_d} computes $\mathbf x^0_j$ ``sub-optimally'' but more efficiently; this explains why Algorithms \ref{proc:init3_order_d} performs better than Algorithm \ref{proc:init4_order_d}. Concerning the efficiency, Algorithms \ref{proc:init1_order_d} and \ref{proc:init2_order_d} require for-loop operations, which is known to be slow in Matlab. This leads to that although the algorithms have similar computational complexity in theory (Algorithms \ref{proc:init2_order_d} and \ref{proc:init3_order_d}), after implementation, their performances are quite different.  

\color{black}
%\begin{figure}[h]
%	\label{fig:d3d4_sparsity_iterative}%% label for second subfigure
%	\includegraphics[height=4cm,width=15cm]{fig/d3d4_sparsity_iterative.eps}
%	\caption{Sparsity level: $\sum^d_{j=1}\bigzeronorm{\mathbf x^k_j}/d$. Left: $d=3$; right: $d=4$.} 
%
%\label{fig:tr1a_sparsity_iterative} 	
%	\end{figure}

\begin{figure}[h] 
\centering
\subfigure[Objective value and CPU time  when $d=3$.]{
	\label{fig:d3_iterative} %% label for second subfigure
	\includegraphics[height=2.5cm,width=13cm]{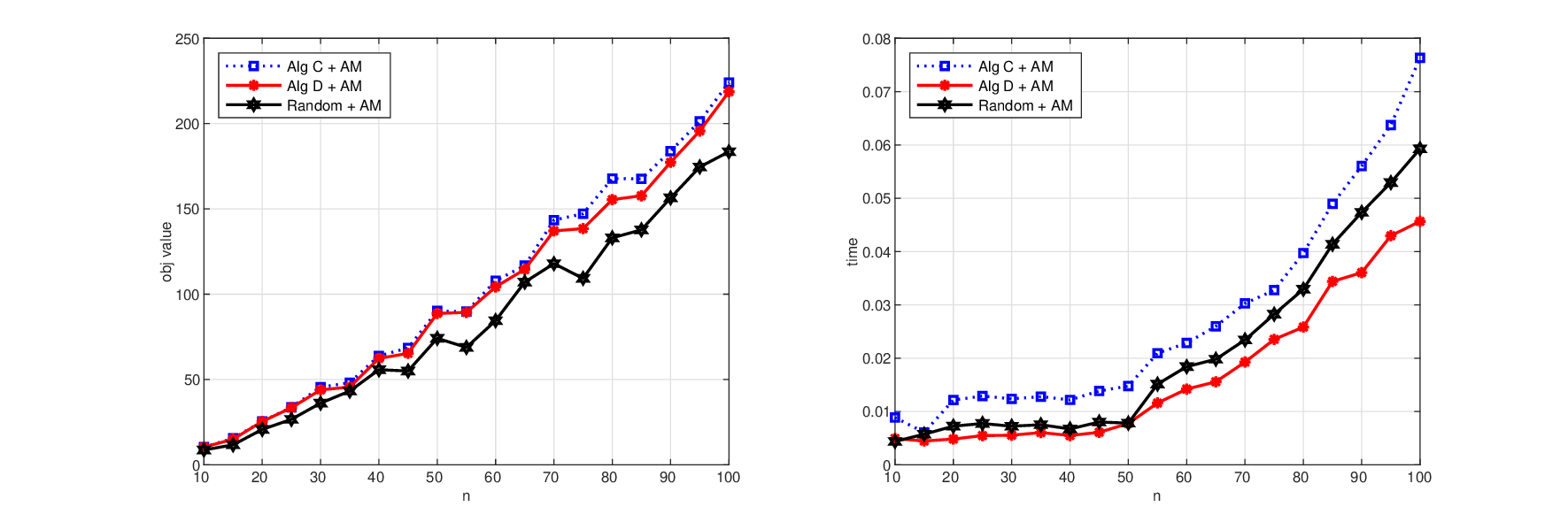}}\\
\subfigure[Objective value and CPU time  when $d=4$.]{
	\label{fig:d4_iterative}%% label for second subfigure
	\includegraphics[height=2.5cm,width=13cm]{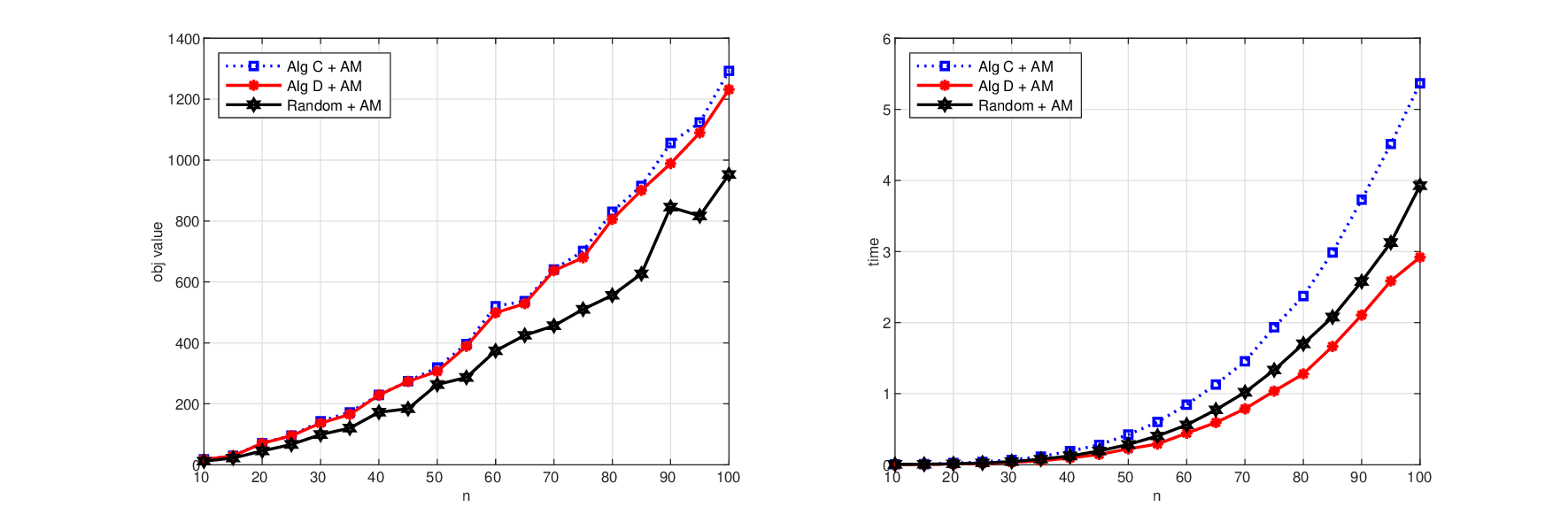}}\\
\subfigure[Number of iterations. Left: $d=3$; right: $d=4$.]{
	\label{fig:d3d4_iterative_iterations}%% label for second subfigure
	\includegraphics[height=2.5cm,width=13cm]{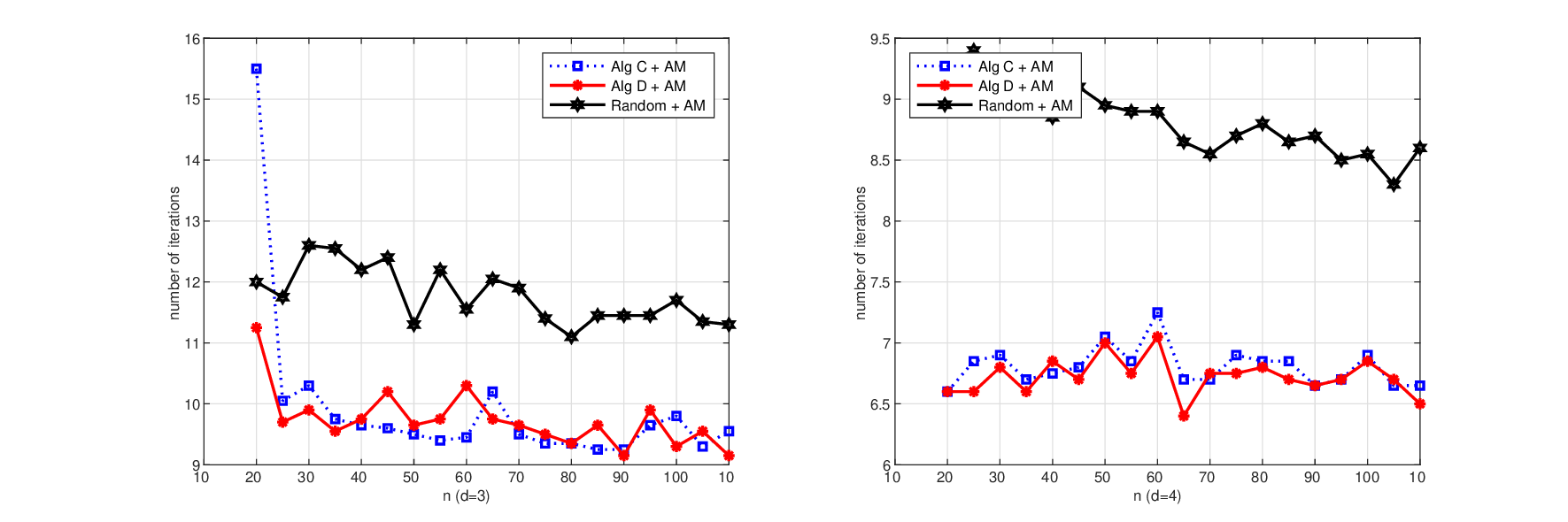}}\\
%    \hspace{0.6in}
%	\subfigure[A $5\times 40\times 40\times 40$ tensor with $R=20$ and $t=3$.]{
	%		\label{fig:compare_dim5_40_40_40_r20}%% label for second subfigure
	%		\includegraphics[height=3.5cm,width=14cm]{fig/3.eps}}    
\caption{Performances of  Alg. \ref{proc:init3_order_d} + AM, Alg. \ref{proc:init4_order_d} + AM and Random + AM for solving \eqref{prob:str1approx_org_max} where $\mathcal A$ is given by \eqref{exp:tr1a}. $n$ varies from $10$ to $100$. The CPU time counts both that of approximation algorithms and  AM. } 

\label{fig:tr1a_iterative} %% label for entire figure
\end{figure}

\paragraph{Performance of approximation plus iterative algorithms}
In this part, we first use approximation algorithms to generate $(\mathbf x^0_1,\ldots,\mathbf x^0_d)$, and then use it as an initializer for iterative algorithms. The goal is to see if approximation algorithms can help in improving the solution quality of iterative algorithms. The iterative algorithm used for solving problem \eqref{prob:str1approx_org_max} is simply an alternating maximization method (termed AM in the sequel) with the scheme
\begin{equation*}
\boxed{  
	{\rm (AM)}~~~~	  \mathbf x^{k+1}_j \in \arg\max \nolimits~  {\mathcal A}{\mathbf x_1^{k+1}\cdots \mathbf x^{k+1}_{j-1}\mathbf x_j \mathbf x^k_{j+1} \cdots \mathbf x_d^k}~{\rm s.t.}~ \bignorm{\mathbf x_j}=1, \bigzeronorm{\mathbf x_j}\leq r_j}
\end{equation*}
for $j=1,\ldots,d$ and $k=1,2,\ldots$. The stopping criterion used for AM is $\max_j\{\bignorm{\mathbf x_j^{k+1} -\mathbf x_j^k}    \}\leq 10^{-5}$, or $k\geq 2000$. We employ Algorithms \ref{proc:init3_order_d} and \ref{proc:init4_order_d} in this part, and denote the approximation plus iterative algorithms as Alg. \ref{proc:init3_order_d} + AM and \ref{proc:init4_order_d} + AM in the sequel. As a baseline, we also evaluate AM initialized by randomly generated feasible point $(\mathbf x^{\rm random}_1,\ldots,\mathbf x^{\rm random}_d)$, which is generated the same as the previous part. This algorithm is denoted as Random + AM. 
The data tensors are also \eqref{exp:tr1a}, where $50$ instances are randomly generated for each $n$. 
$r_j=\lfloor 0.3n_j\rfloor$. The objective values $\mathcal A\mathbf x^{\rm out}_1\cdots\mathbf x^{\rm out}_d$ and CPU time for third- and fourth-order tensors with $n$ varying from $10$ to $100$ are plotted in Fig. \ref{fig:d3_iterative} and \ref{fig:d4_iterative}, where $(\mathbf x^{\rm out}_1,\ldots,\mathbf x^{\rm out}_d)$ is the output of AM. Here the CPU time counts both that of approximation algorithms and  AM. Fig. \ref{fig:d3d4_iterative_iterations} depicts the number of iterations of AM initialized by different strategies.  
Alg. \ref{proc:init3_order_d} + AM is in blue, \ref{proc:init4_order_d} + AM is in red, and Random + AM is in black. 

\begin{figure}[h] 
\centering
\subfigure[$d=3$. Left: $\mathcal A\mathbf x^{\rm out}_1\cdots\mathbf x^{\rm out}_d$ versus $n$; right:   CPU time versus $n$.]{
	\label{fig:d3_iterative_l1} %% label for second subfigure
	\includegraphics[height=2.5cm,width=13cm]{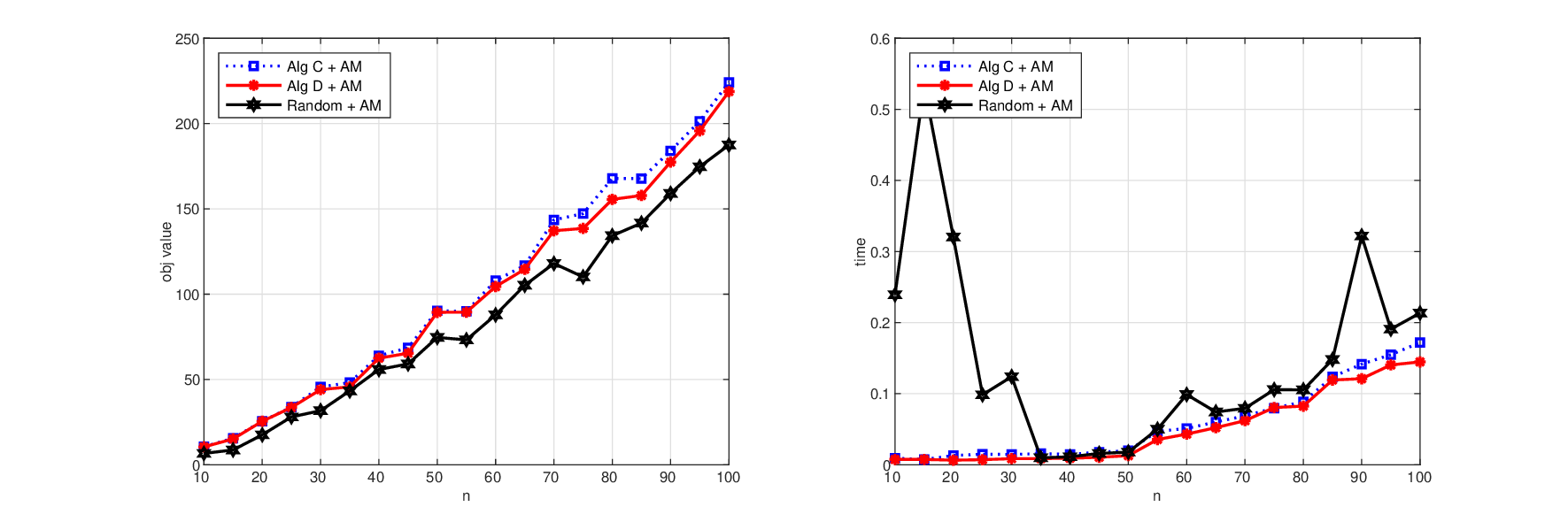}}\\
\subfigure[$d=4$. Left: $\mathcal A\mathbf x^{\rm out}_1\cdots\mathbf x^{\rm out}_d$ versus $n$; right:   CPU time versus $n$.]{
	\label{fig:d4_iterative_l1}%% label for second subfigure
	\includegraphics[height=2.5cm,width=13cm]{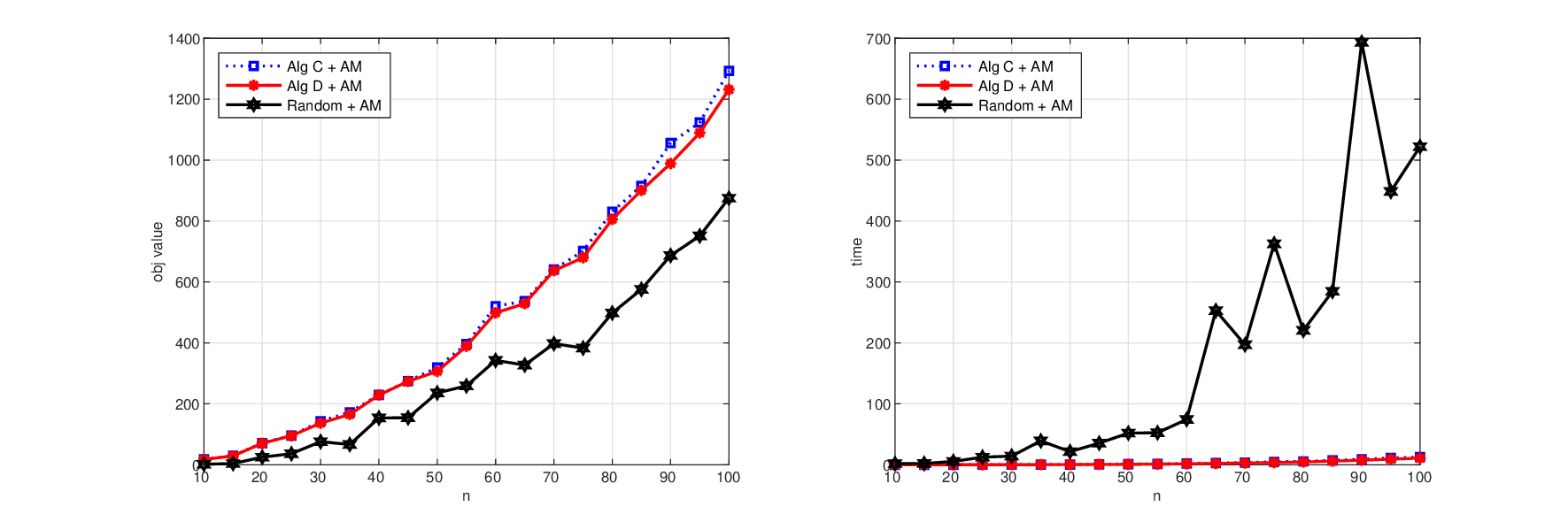}}\\
\subfigure[Number of iterations versus $n$. Left: $d=3$; right: $d=4$.]{
	\label{fig:d3d4_iterative_iterations_l1}%% label for second subfigure
	\includegraphics[height=2.5cm,width=13cm]{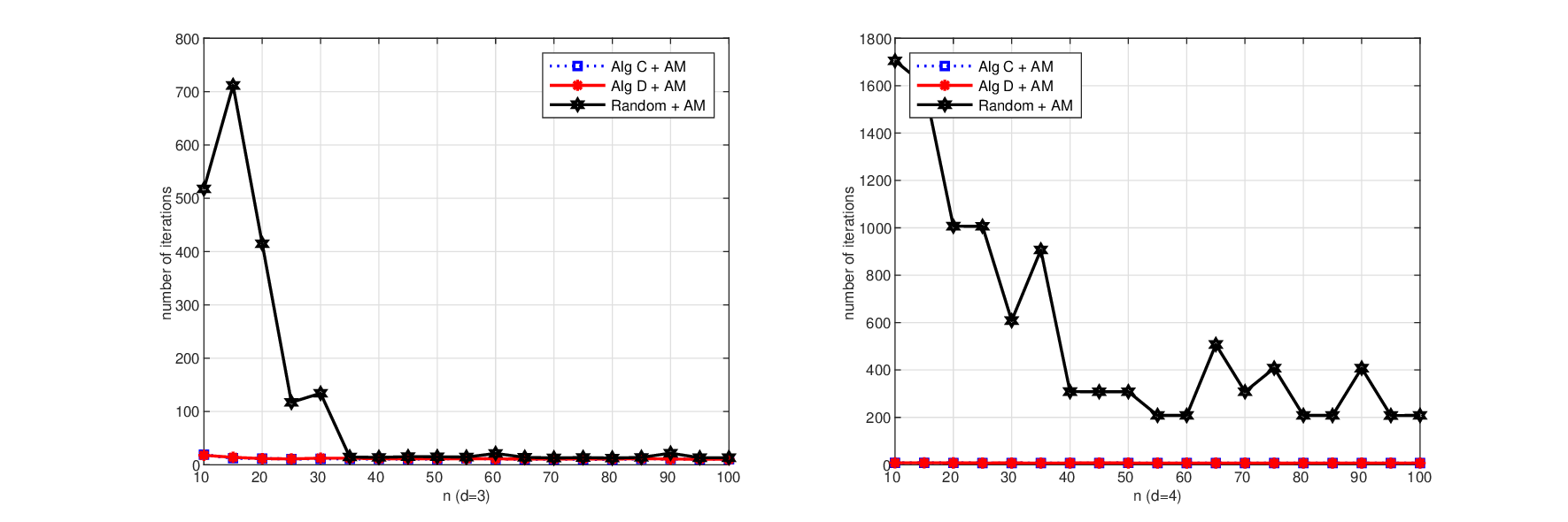}}\\
%    \hspace{0.6in}
%	\subfigure[A $5\times 40\times 40\times 40$ tensor with $R=20$ and $t=3$.]{
	%		\label{fig:compare_dim5_40_40_40_r20}%% label for second subfigure
	%		\includegraphics[height=3.5cm,width=14cm]{fig/3.eps}}    
\caption{Performances of  Alg. \ref{proc:init3_order_d} + AM ($\ell_1$), Alg. \ref{proc:init4_order_d} + AM ($\ell_1$) and Random + AM ($\ell_1$) for solving \eqref{prob:l1} where $\mathcal A$ is given by \eqref{exp:tr1a}. $n$ varies from $10$ to $100$. } 	
\label{fig:tr1a_iterative_l1} %% label for entire figure
\end{figure}

In terms of objective value, we see from Fig. \ref{fig:d3_iterative} and \ref{fig:d4_iterative} that Alg. \ref{proc:init3_order_d} + AM performs the best, while Alg. \ref{proc:init4_order_d} + AM is slightly worse. Both of them are better than Random + AM, demonstrating that     approximation algorithms can indeed help in improving the solution quality of iterative algorithms. Considering the efficiency, Alg. \ref{proc:init4_order_d} + AM is the best among the three, followed by Random + AM. This further demonstrates the advantage of approximation algorithms. In fact, from Fig. \ref{fig:d3d4_iterative_iterations}, we can see that both Alg. \ref{proc:init3_order_d} and \ref{proc:init4_order_d} can help in reducing the number of iterations of AM. However, as we have observed in the previous part, Alg. \ref{proc:init3_order_d} is more time-consuming than Alg. \ref{proc:init4_order_d}, leading to that Alg. \ref{proc:init3_order_d} + AM is the slowest one.

We also try to use Algorithms \ref{proc:init3_order_d} and \ref{proc:init4_order_d} to initialize AM for $\ell_1$ regularized model \cite{allen2012sparse}:
\begin{equation}\label{prob:l1}
\max~\innerprod{\mathcal A}{\mathbf x_1\circ\cdots\circ \mathbf x_d} - \sum^d_{j=1}\nolimits \rho_j\|\mathbf x_j\|_1~~{\rm s.t.}~~\normalnorm{\mathbf x_j} \leq 1,~1\leq j\leq d, 
\end{equation}
where $\rho_j=0.2$ is set in the experiment. The algorithms are denoted as Alg. \ref{proc:init3_order_d} + AM ($\ell_1$), Alg. \ref{proc:init4_order_d} + AM ($\ell_1$)  and Random + AM ($\ell_1$). The results are plotted in Fig. \ref{fig:tr1a_iterative_l1}, from which we observe similar results as those of Fig. \ref{fig:tr1a_iterative}; in particular, the efficiency of approximation algorithms plus AM ($\ell_1$) is significantly better than Random + AM  ($\ell_1$), and their number of iterations is much more stable.

Overall, based on the observations of this part, compared with random initializations, iterative algorithms initialized by approximation solutions generated by approximation algorithms is superior both in terms of 
the solution quality and the running time.

\paragraph{Sparse   tensor clustering}  Tensor clustering  aims to cluster   matrix or  tensor samples into their underlying groups: for $N$
%\footnote{For simplicity of presentation, in the sequel, we denote $n_{d+1} := N$.} 
samples $\mathcal A_i \in \T,i = 1,\dots,N$ with $d\geq 2$ and  given clusters $K\geq 2$, a clustering $\psi (\cdot)$ is defined as a mapping $\psi:\mathbb R^{n_1\times \dots \times n_d}\rightarrow \{1,\dots,K\}$. Several tensor methods have been proposed for tensor clustering; see, e.g., \cite{papalexakis2012k,sun2017provable,sun2019dynamic}. Usually, one can first perform a dimension reduction to the   samples by means of (sparse) tensor decomposition, and then use classic methods such as $K$-means to the reduced samples for clustering. Here, we use a deflation method for tensor decomposition, including the proposed approximation algorithms as initialization procedures for AM as subroutines. 
The whole method for tensor clustering is presented in Algorithm  \ref{alg:sparse_tensor_clustering}, which has some similarities to  \cite[Algorithm 2]{sun2019dynamic}. We respectively use \ref{alg:sparse_tensor_clustering} (\ref{proc:init1_order_d}), \ref{alg:sparse_tensor_clustering} (\ref{proc:init2_order_d}), \ref{alg:sparse_tensor_clustering} (\ref{proc:init3_order_d}), \ref{alg:sparse_tensor_clustering} (\ref{proc:init4_order_d}) to distinguish   AM involved in \ref{alg:sparse_tensor_clustering} initialized by different approximation algorithms. We also denote \ref{alg:sparse_tensor_clustering} (Random) as   AM involved in \ref{alg:sparse_tensor_clustering} with random initializations.

\begin{boxedminipage}{0.97\textwidth}\small
\begin{equation}  \label{alg:sparse_tensor_clustering} \tag{STC}
	\noindent {\rm Algorithm}~ (\psi(\mathcal A_1),\ldots,\psi(\mathcal A_N)) =  {\rm sparse\_clustering}(\mathcal A_1,\dots,\mathcal A_N)
\end{equation}

1. Define $\mathcal T\in R^{n_1\times \cdots \times n_d\times N}$ with $\mathcal T(:,\ldots,:,i) = \mathcal A_i$, $i=1,\ldots,N$.

2. Apply a   deflation method with rank-$R$ to $\mathcal T$, which employs Algorithm  \ref{proc:init1_order_d},  \ref{proc:init2_order_d}, \ref{proc:init3_order_d},  or  \ref{proc:init4_order_d})+ AM    to   the residual tensors  $\mathcal T,\ldots, \mathcal T- \sum^m_{l=1}\alpha_l\mathbf x^l_1\circ\cdots\circ\mathbf x^l_{d+1},\ldots$, with resulting rank-1 terms $\mathbf x^m_1\circ\cdots\circ\mathbf x^m_{d+1}$ and weights $\alpha_m = (\mathcal T - \sum^{m-1}_{  l=1} \alpha_{  l} \mathbf x^{  l}_1\cdots\mathbf x^{  l}_{d+1} )\mathbf x_1^m\cdots\mathbf x^m_{d+1}$,  $m=1,\ldots,R$.    
Write $X_{j} := [\mathbf x_{j}^1,\ldots,\mathbf x_{j}^R]   $, $j = 1,\dots,d+1$ and    $\boldsymbol{\alpha} := [ \alpha_1,\dots,\alpha_R]^\top$.

3. Denote $\hat{X}_{d+1} = [ \alpha_1\mathbf x_{d+1}^1,\dots,\alpha_R\mathbf x_{d+1 }^R]\in \mathbb R^{N\times R}$; here the $i$-th row of $\hat X_{d+1}$, denoted as $\hat{\mathbf u}_i$, is regarded as the reduced sample of $\mathcal A_i$. Apply   $K$-means to clustering  $\{\hat{\mathbf{u}}_1,\dots,\hat{\mathbf{u}}_N\}\rightarrow\hat{\psi}(\hat{\mathbf{u}}_j),j = 1,\dots,N$.

4. Return																																																																																																								 the cluster assignment of $\{\mathcal{A}_1,\dots,\mathcal{A}_N\}$: $(\psi(\mathcal A_1),\ldots,\psi(\mathcal A_N)) = (\hat{\psi}(\hat{\mathbf{u}}_1),\ldots,\hat{\psi}(\hat{\mathbf{u}}_N))$.

\end{boxedminipage}

% The data samples are assigned to a cluster in such a manner that the Euclidean distance between the data sample and the cluster centers would be minimum.
%In order to measure clustering performance, we use two metrics: tensor recovery error and clustering error. The former is defined as:
%
%tensor recovery error = $\frac{||\mathcal{T} - \sum_{r=1}^{R}\alpha^*_r\mathbf x^*_{1,r}\circ\dots\circ\mathbf x^*_{d+1,r}||_F}{||\mathcal{T}||_F}$.

Denote $\psi_0$ as the true clustering mapping and let $|S|$ represent the cardinality of a set $S$. The following metric is commonly used to measure clustering performance \cite{wang2010consistent}:
% In order to measure clustering performance, we use the metric: clustering error, which is defined as the estimated distance between an estimated clustering $ \psi$
% and the true  clustering $\psi_0$ of the sample data:
% 
%\begin{small}
\[
{\rm cluster~ err.} :=  \tbinom{N}{2}^{-1}|\{i,j\}:(\psi(\mathcal A_i)=\psi(\mathcal A_j))\not=(\psi_0(\mathcal A_i)=\psi_0 (\mathcal A_j)),i<j|.
\]
%\end{small}
% where $|A|$ represents the cardinality of the set $A$. This clustering criterion has been commonly used
% in the clustering literature .

\paragraph{Synthetic Data}
This experiment is similar to  \cite[Sect. 5.2]{sun2019dynamic}. 
Consider $N$ sparse matrix samples $\{A_i\in \mathbb{R}^{n_1\times n_2},i = 1,\dots,N\}$ with $n_1=n_2=20$ as follows:

\begin{eqnarray*} 
&&A_1 = \dots = A_{\frac{N}{4}} ~= \mu^3\cdot \begin{bmatrix}
	\Sigma  &~&~ \\ 
	~& -\Sigma  & ~\\ 
	~& ~ & \mathbf 0
\end{bmatrix}, A_{\frac{N}{4}+1} = \dots = A_{\frac{N}{2}} ~= \mu^3\cdot \begin{bmatrix}
	\Sigma  &  ~&~ \\ 
	~& \Sigma  & ~\\ 
	~&~  & \mathbf 0
\end{bmatrix}\\
&&A_{\frac{N}{2}+1} = \dots = A_{\frac{3N}{4}} \hspace{0.3mm}= \mu^3\cdot \begin{bmatrix}
	-\Sigma  &  & \\ 
	& \Sigma  & \\ 
	&  & \mathbf 0
\end{bmatrix}, A_{\frac{3N}{4}+1} = \dots = A_{N} \hspace{0.3mm}= \mu^3\cdot \begin{bmatrix}
	-\Sigma  &  & \\ 
	& -\Sigma  & \\ 
	&  & \mathbf 0
\end{bmatrix},
\end{eqnarray*}	 
where $\Sigma=\mathbf v\mathbf v^{\top} \in\mathbb R^{4\times 4}$ with 
$\mathbf v = [1,-1,0.5,-0.5]^{\top}$,
and $\mathbf 0$ is a zero matrix of size $12\times 12$. 
We 
stack these samples into a  third order tensor $\mathcal{T}\in \mathbb{R}^{n_1\times n_2 \times N}$ with $\mathcal T(:,:,i)=A_i$.
We apply   \ref{alg:sparse_tensor_clustering} (\ref{proc:init1_order_d}), \ref{alg:sparse_tensor_clustering} (\ref{proc:init2_order_d}), \ref{alg:sparse_tensor_clustering} (\ref{proc:init3_order_d}), \ref{alg:sparse_tensor_clustering} (\ref{proc:init4_order_d}),   and the vanilla $K$-means to the   tensor $\mathcal T^* = \frac{\mathcal T}{||\mathcal T||_F} + \sigma\frac{\mathcal E}{||\mathcal E||_F}$    where $\mathcal  E$ is a noisy tensor and $\sigma$ is the noise level. Here vanilla $K$-means stands for directly applying $K$-means to the vectorizations of $A_i$'s   without tensor methods. 
We vary the sample number $N = \{20,40\}$ and noise level $\sigma = \{0.1,0.5,0.9\}$, and set $\mu = 0.5$. 
To select parameters, we apply a similar Bayesian Information Criterion as \cite{allen2012sparse}: For a set of parameter combinations  defined as the cartesian product of the set $ \mathfrak{R}$ of given ranks and the sets $ \mathfrak{r}$ of tuning parameters, namely, $ \mathfrak{R}\times\mathfrak{r} = \{(R,\boldsymbol{r})|R\in \mathfrak{R},\boldsymbol{r}\in \mathfrak{r}\}$, we select  $(  R^*,  {\boldsymbol{r}^*})=\arg\min_{\mathfrak{R}\times\mathfrak{r}} BIC(R,\boldsymbol{r})$ with 
\begin{equation}\label{BIC} \small
BIC(R,\boldsymbol{r}) := \log\left(\frac{\|\mathcal{T} - \sum_{l=1}^{R}\alpha_l\mathbf x^l_{1 }\circ\dots\circ\mathbf x^l_{d+1}\|^2_F}{n_1 \dots n_{d  }N}\right)+\frac{\log(\prod^d_{j=1}n_j N  )}{\prod^d_{j=1}n_j N} \sum_{l = 1}^{R}\nolimits\sum^{d+1}_{j=1}\nolimits\|\mathbf x^l_{j}
\|_0 ,
\end{equation}
where $ {\alpha}_l$ and $\mathbf x^l_j$  are computed by Algorithm \ref{alg:sparse_tensor_clustering}.

For tuning procedure, we set $ \mathfrak{R} = \{4,6\}$ and $\mathfrak{r}= \{(7,7,N),(8,8,N)\}$. 
For  the number $K$ of clusters, we employ the widely used \texttt{evalclusters} function in Matlab which evaluate each proposed number of clusters in a set $\{K_1,\cdots,K_m\}$ and select the smallest number of clusters satisfying
$Gap(K)\geq Gap(K+1) - SE(K+1),$
where $Gap(K)$ is the gap value for the clustering solution with $K$ clusters, and $SE(K + 1)$ is the standard error of the clustering solution with $K + 1$ clusters. The results are reported in Table \ref{tab:2} averaged over 50 instances in each case. 
%recovery err,clustering err,Computational time
\begin{table}[htbp]
\renewcommand{\arraystretch}{2}
\setlength{\tabcolsep}{2pt}
\centering
\caption{Sparse tensor clustering via \ref{alg:sparse_tensor_clustering} (\ref{proc:init1_order_d}), \ref{alg:sparse_tensor_clustering} (\ref{proc:init2_order_d}), \ref{alg:sparse_tensor_clustering} (\ref{proc:init3_order_d}), \ref{alg:sparse_tensor_clustering} (\ref{proc:init4_order_d})  on synthetic data with different size $N$ and noise level $\sigma$. ``cluster err.'' represents the cluster error; ``time'' denotes the computational time (seconds). %The cases with least cluster error are marked in bold.
}
\begin{mytabular} {cccccccccccc}       
	\toprule
	& & \multicolumn{2}{c}{\ref{alg:sparse_tensor_clustering} (\ref{proc:init1_order_d})} & \multicolumn{2}{c}{\ref{alg:sparse_tensor_clustering} (\ref{proc:init2_order_d})} & \multicolumn{2}{c}{\ref{alg:sparse_tensor_clustering} (\ref{proc:init3_order_d})} &
	\multicolumn{2}{c}{\ref{alg:sparse_tensor_clustering} (\ref{proc:init4_order_d})} & \multicolumn{2}{c}{vanilla $K$-means}\\
	\cmidrule(r){3-4} \cmidrule(r){5-6} \cmidrule(r){7-8} \cmidrule(r){9-10} \cmidrule(r){11-12} 
	$N$ &   $\sigma$ & cluster err.  & time & cluster err. & time & cluster err. & time & cluster err. & time & cluster err. & time\\
	\toprule 
	20    & 0.1   & \bf{0.00E+00} & 2.72E-01 & \bf{0.00E+00} & 5.08E-01 & \bf{0.00E+00} & 2.41E-01 & \bf{0.00E+00} & 2.08E-01 & 9.37E-03 & 2.10E-03 \\
	20    & 0.5   & 3.05E-03 & 3.03E-01 & 6.53E-03 & 5.47E-01 & \bf{0.00E+00} & 2.83E-01 & \bf{0.00E+00} & 2.27E-01 & 2.53E-02 & 1.82E-03 \\
	20    & 0.9   & \bf{1.22E-02} & 3.72E-01 & 6.53E-03 & 5.17E-01 & \bf{1.22E-02} & 2.83E-01 & 3.05E-03 & 2.17E-01 & 5.21E-02 & 1.82E-03 \\
	40    & 0.1   & \bf{0.00E+00} & 3.98E-01 & \bf{0.00E+00} & 7.10E-01 & \bf{0.00E+00} & 4.26E-01 & \bf{0.00E+00} & 5.36E-01 & 3.10E-03 & 2.35E-03 \\
	40    & 0.5   & 3.18E-03 & 4.26E-01 & 6.28E-03 & 6.43E-01 & \bf{0.00E+00} & 3.85E-01 & 2.97E-03 & 3.54E-01 & 2.26E-02 & 2.39E-03 \\
	40    & 0.9   & 3.21E-03 & 3.77E-01 & 9.08E-03 & 6.69E-01 & \bf{3.10E-03} & 3.57E-01 & 3.18E-03 & 3.13E-01 & 5.33E-02 & 2.50E-03 \\
	\bottomrule
\end{mytabular}
\label{tab:2}
\end{table}

Table \ref{tab:2} shows that the cluster error of \ref{alg:sparse_tensor_clustering} with any approximation algorithm is smaller than that of the vanilla $K$-means in all cases and  is even zero in some case when the noise level is not high,  where the best one is \ref{alg:sparse_tensor_clustering} (\ref{proc:init3_order_d}), followed by \ref{alg:sparse_tensor_clustering} (\ref{proc:init4_order_d}).   This shows the accuracy and robustness of our method. Considering the CPU time, among the four tensor methods, \ref{alg:sparse_tensor_clustering} (\ref{proc:init4_order_d}) is the most efficient one, followed by \ref{alg:sparse_tensor_clustering} (\ref{proc:init3_order_d}), while \ref{alg:sparse_tensor_clustering} (\ref{proc:init2_order_d}) needs more time than the other three. 
However, the computational time  of tensor methods is not as good as that of the vectorized $K$-means, which is because of the increasing cost of the tuning procedure via  \eqref{BIC} with a pre-specified set of parameter combinations.
\paragraph{Real data}
We test the clustering performance of \ref{alg:sparse_tensor_clustering} (\ref{proc:init4_order_d}), \ref{alg:sparse_tensor_clustering} (Random), and the vanilla $K$-means 
on Columbia Object Image Library COIL-20 \cite{nene1996columbia} for image clustering. The   data contains 20 objects viewed from varying angles and each image is of size $128\times 128$. The images are in grayscale, and the background of the objects is black, resulting in that the images can be seen as sparse matrices.  In this experiment, we consider $K = \{5, 6, 7, 8, 9, 10\}$ objects and pick up 36 images from each object for clustering, giving $\mathcal T\in \mathbb R^{128\times 128\times 36K}$. 
We still tune  parameters of Algorithm \ref{alg:sparse_tensor_clustering} via  \eqref{BIC} and the  \texttt{evalclusters}  function in Matlab, and set $ \mathfrak{R} = \{40,45\}$ and $\mathfrak{r}= \{(75,75,36K),(80,80,36K)\}$.
The experiment has been repeated 20 times for each $K$, and the averaged results are shown in Table \ref{tab:3}.

\begin{table}[htbp]
\renewcommand{\arraystretch}{2}
\setlength{\tabcolsep}{10pt}
\centering
\caption{Sparse tensor clustering via \ref{alg:sparse_tensor_clustering} (\ref{proc:init4_order_d}), \ref{alg:sparse_tensor_clustering} (Random)  and  vanilla $K$-means on COIL-20 with different   $K$. %``cluster err.'' denotes the cluster error; ``time'' denotes the computational time (seconds). The cases with least cluster error are marked in bold.
}
\begin{mytabular} {ccccccc}       
	\toprule
	&  \multicolumn{2}{c}{\ref{alg:sparse_tensor_clustering} (\ref{proc:init4_order_d})} & \multicolumn{2}{c}{\ref{alg:sparse_tensor_clustering} (Random)} &  \multicolumn{2}{c}{vanilla $K$-means}\\
	\cmidrule(r){2-3} \cmidrule(r){4-5} \cmidrule(r){6-7} 
	$K$ & cluster err.  & time  & cluster err. & time.& cluster err. & time \\
	\toprule 
	5     & \bf{1.36E-01} & 1.62E+02 & 1.37E-01 & 1.70E+02 & 1.46E-01 & 2.21E-01 \\
	6     & \bf{1.45E-01} & 1.90E+02 & 1.50E-01 & 2.02E+02 & 1.51E-01 & 2.79E-01 \\
	7     & \bf{1.47E-01} & 2.23E+02 & 1.49E-01 & 2.46E+02 & 1.56E-01 & 3.53E-01 \\
	8     & \bf{1.48E-01} & 2.63E+02 & 1.49E-01 & 2.84E+02 & 1.55E-01 & 3.91E-01 \\
	9     & \bf{1.48E-01} & 2.93E+02 & 1.49E-01 & 2.97E+02 & 1.57E-01 & 4.40E-01 \\
	10    & 1.13E-01 & 3.09E+02 & \bf{1.11E-01} & 3.43E+02 & 1.24E-01 & 5.39E-01 \\

	\bottomrule
\end{mytabular}
\label{tab:3}
\end{table}

Table \ref{tab:3} shows that the   cluster error   of tensor methods is still better than that of the vanilla $K$-means, while \ref{alg:sparse_tensor_clustering} \eqref{proc:init4_order_d} is slightly better than \ref{alg:sparse_tensor_clustering} (Random). On the other hand, the computational time of \ref{alg:sparse_tensor_clustering} \eqref{proc:init4_order_d} is less than that of \ref{alg:sparse_tensor_clustering} (Random). However, it still needs more time   than that of the vanilla $K$-means. There are two possible reasons for this phenomenon: first, it   takes more time to tune  parameters via \eqref{BIC}; and second, our approach deals directly with data tensors  instead of reshaping them into vectors, which preserves their intrinsic structures but naturally increases the computational complexity at the same time. 

Overall, tensor methods armed with an   approximation algorithm to generate good initializers show its  ability to better exploring the data structure for clustering, and it would be interesting to further improve the efficiency. 

\section{Conclusions}\label{sec:conclusions}
Sparse tensor BR1Approx problem can be seen as a sparse generalization of the dense tensor BR1Approx problem and a higher-order extension of the matrix BR1Approx problem. In the literature, little attention was paid to approximation algorithms for sparse tensor BR1Approx. To fill this gap, four  approximation algorithms were developed in this work, which are of low computational complexity, easily implemented, and all admit theoretical guaranteed %worst-case 
approximation bounds. Some of the proposed algorithms and the associated aproximation bounds generalize their matrix or dense tensor counterparts, while Algorithm \ref{proc:init4_order_d}, which is the most efficient one, is even new when reducing to the matrix or dense tensor cases. Numerical experiments on synthetic as well as real data showed the effectiveness and efficiency of the developed algorithms; in particular, we observed that compared with random initializations, our algorithms can improve the performance of  iterative algorithms for solving the problem in question. Possible future work is to design algorithms with better approximation ratio; following \cite[Theorem 4.3]{he2012probability}, we conjecture that the best ratio  might be $O(\sqrt{\prod^{d}_{j=1}\frac{r_j}{n_j}} \cdot  \sqrt{\prod^{d-2}_{j=1} \frac{\ln n_j  }{ n_j} }   )$.  On the other hand, Qi defined and studied the best rank-1 approximation ratio of a tensor space \cite{qi2011best}; it would be interesting to extend this notion to the sparse best rank-1 approximation setting and study its bounds.

 {\scriptsize\section*{Acknowledgement}  This work was supported by the National Natural  Science Foundation of China  Grant 11801100, the    Fok Ying Tong Education Foundation Grant 171094, and the  Innovation Project of Guangxi Graduate Education Grant YCSW2020055.}

\bibliographystyle{plain}
\bibliography{../../tensor,../../robust,../../sparse_pca}

\end{document}